\documentclass[10pt]{amsart}
\allowdisplaybreaks

\usepackage{amsfonts}
\usepackage{amsmath}
\usepackage{amssymb}
\usepackage{amsthm}
\usepackage{bbm}
\usepackage{graphicx} \usepackage{enumerate} \usepackage{multicol}
\usepackage{mathrsfs} \usepackage[all,cmtip]{xy}
\usepackage{xcolor}

\newcommand\transpose{%
  {\mathchoice
    {\raisebox{.45ex}{$\displaystyle{\intercal}$}}
    {\raisebox{.45ex}{$\textstyle{\intercal}$}}
    {\raisebox{.30ex}{$\scriptstyle{\intercal}$}}  
    {\raisebox{.23ex}{$\scriptscriptstyle{\intercal}$}}}
  }

\newcounter{dummy} \numberwithin{dummy}{section}
\newtheorem{theorem}[dummy]{Theorem}
\newtheorem{corollary}[dummy]{Corollary}
\newtheorem{lemma}[dummy]{Lemma}

\newtheorem{proposition}[dummy]{Proposition}
\theoremstyle{remark}
\newtheorem{remark}[dummy]{Remark}
\newtheorem{example}[dummy]{Example}

\newcommand{\calC}{\mathcal C}
\newcommand{\calF}{\mathcal F}
\newcommand{\calR}{\mathcal R}
\newcommand{\calL}{\mathcal L}

\DeclareMathOperator{\Ann}{Ann}
\DeclareMathOperator{\Sym}{Sym}
\DeclareMathOperator{\End}{End}
\DeclareMathOperator{\Dom}{Dom}
\DeclareMathOperator{\spn}{span}
\DeclareMathOperator{\pr}{pr}
\DeclareMathOperator{\id}{id}
\DeclareMathOperator{\tr}{tr}

\DeclareMathOperator{\dv}{div}
\DeclareMathOperator{\Ric}{\mathrm{Ric}}
\newcommand{\rnabla}{\mathring{\nabla}}
\newcommand{\shh}{\sharp^H}

\newcommand\blank{{\kern.8pt\displaystyle\cdot\kern.8pt}}
\newcommand\blankdown{{\cdot\kern.8pt}}

\DeclareMathOperator{\Ort}{O}

\DeclareMathOperator{\ad}{ad}

\DeclareMathOperator{\II}{\mathit{I\!I}}

\DeclareMathOperator{\Lie}{Lie}
\newcommand{\ptr}{/\!/}
\newcommand{\hptr}{/ \! \hat{/}}
\newcommand{\ve}{\varepsilon}
\DeclareMathOperator{\dilation}{Dil}
\DeclareMathOperator{\Cov}{Cov}

\numberwithin{equation}{section}

\title[Stochastic completeness and gradient representations]{Stochastic completeness and gradient representations for sub-Riemannian manifolds}
\author[E. Grong, A. Thalmaier]{Erlend Grong and Anton Thalmaier }

\address{Universit\'e Paris-Sud,  
3, rue Joliot-Curie, 91192 Gif-sur-Yvette, France, and \newline
University of Bergen, Department of Mathematics,  P.O. Box 7803, 5020 Bergen, Norway}
\email{erlend.grong@math.uib.no}

\address{Mathematics Research Unit, FSTC, University of Luxembourg,
Maison du Nombre, 6, avenue de la Fonte, L-4364 Esch-sur-Alzette, Luxembourg}
\email{anton.thalmaier@uni.lu}

\thanks{This work has been supported by the Fonds
National de la Recherche Luxembourg (FNR) under the OPEN scheme
(project GEOMREV O14/7628746). The first author supported by project 249980/F20 of the Norwegian Research Council.} 



\begin{document}

\begin{abstract}
Given a second order partial differential operator $L$ satisfying the strong H\"ormander condition with corresponding heat semigroup $P_t$, we give two different stochastic representations of $dP_t f$ for a bounded smooth function~$f$. We show that the first identity can be used to prove infinite lifetime of a diffusion of $\frac{1}{2} L$, while the second one is used to find an explicit pointwise bound for the horizontal gradient on a Carnot group. In both cases, the underlying idea is to consider the interplay between sub-Riemannian geometry and connections compatible with this geometry. 
\end{abstract}

\maketitle

\section{Introduction}
A Brownian motion on a Riemannian manifold $(M, g)$ is a diffusion process with infinitesimal generator equal to one-half of the Laplace-Beltrami operator~$\Delta_{g}$ on $M$. If $(M, g)$ is a complete Riemannian manifold, a lower bound for the Ricci curvature is a sufficient condition for Brownian motion to have infinite lifetime~\cite{Yau78}. Stated in terms of the minimal heat kernel $p_t(x,y)$ to $\frac{1}{2} \Delta_{g}$, this means that $\int_M p_t( x, y) \, d\mu(y) = 1$ for any
$(t,x) \in (0, \infty) \times M$, where $\mu =\mu_g$ is the Riemann volume density.
Infinite lifetime of the Brownian motion is equivalent to uniqueness of solutions to the heat equation in $L^\infty$, see e.g.~\cite{Gri13}, \cite[Section~5]{Has60}. Furthermore, let $P_t$ denote the minimal heat semigroup of $\frac{1}{2} \Delta_{g}$ and let $\nabla f$ denote the gradient of a smooth function with respect to $g$. Then a lower Ricci bound also guarantees that $t \mapsto \| \nabla P_t f \|_{L^\infty(g)}$ is bounded on any finite interval whenever $\nabla f$ is bounded. This fact allows one to use the $\Gamma_2$-calculus of Bakry-\'Emery, see e.g.~\cite{BaEm85,BaLe96}.

For any second order partial differential operator $L$ on $M$, let $\sigma(L) \in \Gamma(\Sym^2 TM)$ denote its symbol, i.e.~the symmetric, bilinear tensor on the cotangent bundle $T^*M$ uniquely determined by the relation
\begin{equation} \label{symbol} \sigma(L)(df, d\phi) = \frac{1}{2} \left( L (f\phi) - f L\phi - \phi Lf \right), \quad f,\phi \in C^\infty(M).\end{equation}
If $L$ is elliptic, then $\sigma(L)$ coincides with the cometric $g^*$ of some Riemannian metric $g$ and $L$ can be written as $L = \Delta_{g} +Z$ for some vector field $Z$. Hence, we can use the geometry of $g$ along with the vector field $Z$ to study the properties of the heat flow of $L$, see e.g. \cite{Wan04}. If $\sigma(L)$ is only positive semi-definite we can still associate a geometric structure known as a sub-Riemannian structure. Recently, several results have appeared linking sub-Riemannian geometric invariants to properties of diffusions of corresponding second order operators  and their heat semigroup, see \cite{BaBo12,BBG14,BKW16,GrTh14a,GrTh14b}. These results are based on a generalization of the $\Gamma_2$-calculus for sub-Riemannian manifolds, first introduced in \cite{BaGa17}. As in the Riemannian case, the preliminary requirements for using this $\Gamma_2$-calculus is that the diffusion of $L$ has infinite lifetime and that the gradient of a function does not become unbounded under the application of the heat semigroup.

Consider the following example of an operator $L$ with positive semi-definite symbol. Let $(M, g)$ be a complete Riemannian manifold with a foliation $\mathcal{F}$ corresponding to an integrable distribution $V$. Let $H$ be the orthogonal complement of $V$ with corresponding orthogonal projection $\pr_{H}$ and define a second order operator~$L$ on~$M$ by
\begin{equation} \label{presrL} L f = \mathrm{div} \, (\pr_{H} \nabla f),   \quad f \in C^\infty(M).\end{equation}

If $H$ satisfies the bracket-generating condition, meaning that the sections of~$H$ along with their iterated brackets span the entire tangent bundle, then $L$ is a hypoelliptic operator by H\"ormander's classical theorem \cite{Hor67}. The operator $L$ corresponds to the sub-Riemannian metric $g_H = g|{H}$. Let us make the additional assumption that leaves of the foliation are totally geodesic submanifolds of $M$ and that the foliation is Riemannian. If only the first order brackets are needed to span the entire tangent bundle, it is known that any $\frac{1}{2} L$-diffusion $X_t$ has infinite lifetime given certain curvature bounds \cite[Theorem 3.4]{GrTh14b}. Furthermore, if $H$ satisfies the Yang-Mills condition, then no assumption on the number of brackets needed to span the tangent bundle is necessary \cite[Section~4]{BKW16}, see Remark~\ref{re:geometric} for the definition of the Yang-Mills condition. Under the same restrictions, for any smooth function $f$ with bounded gradient, $t \mapsto \|\nabla P_tf\|_{L^\infty(g)}$ remains bounded on a finite interval.

We will show how to modify the argument in \cite{BKW16} to go beyond the requirement of the Yang-Mills condition and even beyond foliations. We will start with some preliminaries on sub-Riemannian manifolds and sub-Laplacians in Section~\ref{sec:sR}. In Section~\ref{sec:Weitz} we will show that the existence of a Weitzenb\"ock-type formula for a connection sub-Laplacian is always corresponds to adjoint of a connection compatible with a sub-Riemannian structure. Our results on infinite lifetime are presented in Section~\ref{sec:InfiniteLifetime} based on a Feynman-Kac representation of~$dP_tf$ using a particular adjoint of a compatible connection. Using recent results of~\cite{Elw14}, we also show that our curvature requirement in the case of totally geodesic foliations implies that the Brownian motion of the full Riemannian metric $g$ has infinite lifetime as well, see Section~\ref{sec:FoliationCounter}.

Our Feynman-Kac representation in Section~\ref{sec:InfiniteLifetime} uses parallel transport with respect to a connection that does not preserve the horizontal bundle.  We give an alternative stochastic representation if $dP_tf$ using parallel transport along a connection that preserves our sub-Riemannian structure in Section~\ref{sec:TorsionRep}. This rewritten representation allows us to give an explicit pointwise bound for the horizontal gradient in Carnot groups. For a smooth function $f$ on $M$, the horizontal gradient $\nabla^H f$ is defined by the condition that $\alpha(\nabla^H f) = \sigma(L)(df, \alpha)$ for any $\alpha \in T^*M$. Carnot groups  are the `flat model spaces' in sub-Riemannian geometry in the sense that their role is similar to that of Euclidean spaces in Riemannian geometry. See Section~\ref{sec:Carnot} for the definition. It is known that there exists pointwise bounds for the horizontal gradient on Carnot groups. From \cite{Mel08}, there exists constants $C_p$ such that
\begin{equation} \label{pointwise} | \nabla^H P_t f|_{g_H} \leq C_p \left( P_t | \nabla^H f |^p_{g_H} \right)^{1/p}, \quad p \in (1, \infty), \end{equation}
holds pointwise for any $t >0$. The constant $C_p$ has to be strictly larger than~$1$, see \cite{DrMe05}. We give explicit constants the gradient estimates on Carnot groups. This result improves on the constant found in \cite{BBBC08} for the special case of the Heisenberg group. Also, for $p >2$ we find a constant that does not depend on the heat kernel.

Appendix~\ref{sec:FK} deals with Feynman-Kac representations of semigroups whose generators are not necessarily self-adjoint, which is needed for the result in Section~\ref{sec:InfiniteLifetime}.


\section{Sub-Riemannian manifolds and sub-Laplacians} \label{sec:sR}
\subsection{Sub-Riemannian manifolds} We define \emph{a sub-Riemannian manifold} as a triple $(M, H, g_H)$ where $M$ is a connected manifold, $H\subseteq TM$ is a subbundle of the tangent bundle and $g_H$ is a metric tensor defined only on $H$. Such a structure induces a map $\shh: T^* M \to H \subseteq TM$ by the formula
\begin{equation} \label{shh} \alpha(v) =g_H(\shh \alpha, v) = : \langle \shh \alpha, v \rangle_{g_H} \quad \alpha \in T_x M,\ v \in H_x,\ x \in M.\end{equation}
The kernel of this map is the subbundle $\Ann(H) \subseteq T^*M$ of covectors vanishing on~$H$. This map $\sharp^H$ induces a cometric $g_H^*$ on $T^*M$ by the formula
\begin{equation} \label{metrictocometric}\langle \alpha, \beta \rangle_{g_H^*} = \langle \sharp^H \alpha, \sharp^H \beta \rangle_{g_H},\end{equation}
which is degenerate unless $H = TM$. Conversely, a given cometric $g_H^*$ degenerating along a subbundle of $T^*M$, we can define $\shh \alpha = g_H^*(\alpha, \blank)$ and use \eqref{metrictocometric} to obtain~$g_H$. Going forward, we will refer to $g_H^*$ and $(H, g_H)$ interchangeably as a \emph{sub-Riemannian structure} on $M$. We will call $H$ \emph{the horizontal bundle}. For the rest of the paper, $n$ is the rank of $H$ while $n+\nu$ denotes the dimension of $M$.

Let $\mu$ be a chosen smooth volume density with corresponding divergence $\dv_\mu$. Relative to $\mu$, we can define a second order operator
\begin{equation} \label{srL} \Delta_{H} f := \Delta_{g_H}f = \mathrm{div}_\mu \, \shh df. \end{equation}
By means of definition \eqref{symbol}, the symbol of $\Delta_{H}$ satisfies $\sigma(\Delta_{H}) = g_H^*$. Locally the operator $\Delta_{H}$ can be written as
$$\Delta_{H} f = \sum_{i=1}^n A_i^2 f + A_0f, \quad n = \mathrm{rank} \, H,$$
where $A_0, A_1, \dots, A_n$ are vector fields taking values in $H$ such that $A_1,\dots,A_n$ form a local orthonormal basis of $H$.

The horizontal bundle $H$ is called \emph{bracket-generating} if the sections of $H$ along with its iterated brackets span the entire tangent bundle. The horizontal bundle is said to have \emph{step} $k$ at $x$ if $k-1$ is the minimal order of iterated brackets needed to span $T_xM$. From the local expression of $\Delta_{H}$, it follows that $H$ is bracket-generating if and only if $\Delta_{H}$ satisfies \emph{the strong H\"ormander condition} \cite{Hor67}. We shall assume that this condition indeed holds, giving us that both $\Delta_{H}$ and $\frac{1}{2} \Delta_H - \partial_t$ are hypoelliptic and that
\begin{equation} \label{distance} \mathsf{d}_{g_H}(x,y) := \sup \left\{|f(x) - f(y)| \, \colon \ f \in C_c^\infty(M),\ \sigma(\Delta_{H})(df, df) \leq 1 \right\},\end{equation}
is a well defined distance on $M$. Here, and in the rest of the paper, $C_c^\infty(M)$ denotes the smooth, compactly supported functions on $M$. Alternatively, the distance $\mathsf{d}_{g_H}(x,y)$ can be realized as the infimum of the lengths of all absolutely continuous curves tangent to $H$ and connecting $x$ and~$y$. The bracket-generating condition ensures that such curves always exist between any pair of points. For more information on sub-Riemannian manifolds, we refer to \cite{Mon02}.

In what follows, we will always assume that $H$ is bracket-generating, unless otherwise stated explicitly. We note that if $\Delta_{H}$ satisfies the strong H\"ormander condition and if $\mathsf{d}_{g_H}$ is a complete metric, then $\Delta_{H} |{C_c^\infty(M)}$ is essentially self-adjoint by \cite[Chapter~12]{Str86}.

For the remainder of the paper, we make the following notational conventions. If $p:E \to M$ is a vector bundle, we denote by $\Gamma(E)$ the space of smooth sections of~$E$. If $E$ is equipped with a connection~$\nabla$ or a (possibly degenerate) metric tensor~$g$, we denote the induced connections on $E^*$, $\bigwedge^2 E$, etc. by the same symbol, while the induced metric tensors are denoted by $g^*$, $\wedge^2 g$, etc. For elements $e_1, e_2$, we write $g( e_1, e_2 ) = \langle e_1, e_2\rangle_g$ and $|e_1|_g = \langle e_1, e_1 \rangle^{1/2}_g$ even in the cases when $g$ is only positive semi-definite. If $\mu$ is a chosen volume density on $M$ and $f$ is a function on $M$, we write $\| f\|_{L^p}$ for the corresponding $L^p$-norm with the volume density being implicit. If $Z \in \Gamma(E)$ then $\| Z\|_{L^p(g)} := \| | Z|_g \|_{L^p}$.

For $x \in M$, if $\mathscr{A} \in \End T_xM$ is an endomorphism, we will let $\mathscr{A}^\transpose \in \End T_x^* M$ denote its transpose. If $g$ has a Riemannian metric, then $\mathscr{A}^* \in \End T_x^*M$ denotes its dual. In other words,
$$\langle \mathscr{A} v, w\rangle_{g} = \langle v, \mathscr{A}^* w\rangle_{g}, \quad (\mathscr{A}^\transpose \alpha)(v) = \alpha(\mathscr{A} v), \quad \alpha \in T_x^*M,\quad v, w \in T_xM.$$
The same conventions apply for endomorphisms of $T^*M$. If $\mathscr{A}$ is a differential operator, then $\mathscr{A}^*$ is defined with respect to the $L^2$-inner product of $g$.

\subsection{Taming metrics} Given a sub-Riemannian manifold $(M, H, g_H)$, a Riemannian metric $g$ on $M$ is said to \emph{tame} $g_H$ if $g|H = g_H$. If $\mathsf{d}_{g}$ is the corresponding Riemannian distance, then $\mathsf{d}_{g}(x,y) \leq \mathsf{d}_{g_H}(x,y)$ for any $x,y \in M$, since curves tangent to $H$ have equal length with respect to both metrics, while $\mathsf{d}_{g}$ considers the infimum of the lengths over curves that are not tangent to $H$ as well. It follows that if $\mathsf{d}_{g}$ is complete, then $\mathsf{d}_{g_H}$ is a complete metric as well, as observed in \cite[Theorem~7]{Str86}. By \cite[Theorem~2.4]{Str83}, if~$g$ is a complete Riemannian metric taming~$g_H$, then the sub-Laplacian $\Delta_{H}$ with respect to the volume density of $g$ and the Laplace-Beltrami operator $\Delta_g$ are both essentially self adjoint on $C^\infty_c(M)$.

We then denote the corresponding orthogonal projection to $H$ by $\pr_{H}$. Let $\flat:TM \to T^*M$ be the vector bundle isomorphism $v \mapsto \langle v, \blank \rangle_{g}$ with inverse $\sharp$. The fact that $g$ tames~$g_H$ is equivalent to the statement $\shh = \pr_{H} \sharp$. Let $V$ denote the orthogonal complement of $H$ with corresponding projection.
\emph{The curvature} $\calR$ and \emph{the cocurvature} $\bar{\calR}$ of $H$ with respect to the complement $V$ are defined as
\begin{equation} \label{CurvCocurv} \calR(A,Z) = \pr_{V} [\pr_{H} A, \pr_{H} Z], \quad \bar{\calR}(A,Z) = \pr_{H} [\pr_{V} A, \pr_{V} Z],
\end{equation}
for $A, Z \in \Gamma(TM)$. By definition, $\calR$ and $\bar{\calR}$ are vector-valued two-forms, and $\bar{\calR}$ vanishes if and only if $V$ is integrable. The curvature and the cocurvature only depend on the direct sum $TM = H \oplus V$ and not the metrics $g_H$ or $g$.

\subsection{Connections compatible with the metric} \label{sec:Compatible}
Let $\nabla$ be an affine connection on $TM$. We say that $\nabla$ is \text{compatible} with the sub-Riemannian structure $(H, g_H)$ or $g_H^*$ if $\nabla g_H^* =0$. 
This condition is equivalent to requiring that $\nabla$ preserves the horizontal bundle $H$ under parallel transport and that $Z \langle A_1, A_2\rangle_{g_H} = \langle  \nabla_Z A_1, A_2 \rangle_{g_H} + \langle A_1, \nabla_Z A_2\rangle_{g_H}$ for any $Z \in \Gamma(TM)$, $A_1, A_2 \in \Gamma(H)$. For any sub-Riemannian manifold $(M, H, g_H)$, the set of compatible connections is non-empty. Let $\tilde g$ be any Riemannian metric on $M$ and define $V$ as the orthogonal complement to $H$. Let $\pr_{H}$ and $\pr_{V}$ be the corresponding orthonormal projections. Define
$$g = \pr_{H}^* g_H + \pr_{V}^* \tilde g|{V}.$$
Then $g$ is a metric taming $g_H$. Let $\nabla^g$ be the Levi-Civita connection of $g$ and define finally
\begin{equation} \label{Exists} \nabla^0 := \pr_{H}  \nabla^g \pr_{H} + \pr_{V} \nabla^g \pr_{V}.\end{equation}
The connection $\nabla^0$ will be compatible with $g_H^*$ and also with $g$.

\subsection{Rough sub-Laplacians}
In this section, we will introduce rough sub-Laplacians and compare them to the sub-Laplacian as defined in \eqref{srL}. Let $g_H^* \in \Gamma(\Sym^2 TM)$ be a sub-Riemannian structure on $M$ with horizontal bundle $H$. For any two-tensor $\xi \in \Gamma(T^*M^{\otimes 2})$ we write $\tr_{H} \xi(\times, \times) := \xi(g_H^*)$. We use this notation since for any $x\in M$ and any orthonormal basis $v_1, \dots, v_n$ of $H_x$
$$\tr_H \zeta(x)(\times, \times) = \sum_{i=1}^n \zeta(x)(v_i, v_i).$$
For any affine connection $\nabla$ on $TM$, define the Hessian $\nabla^2$ by
$$\nabla^2_{A, B} = \nabla_A \nabla_B - \nabla_{\nabla_A B}.$$
We define \emph{the rough sub-Laplacian} $L(\nabla)$ as $L(\nabla) = \tr_H \nabla^2_{\times, \times}$. Since $\nabla$ induces a connection on all tensor bundles, $L(\nabla)$ is defines as an operator on tensors in general. We have the following result.

\begin{lemma} \label{lemma:DualSL}
\begin{enumerate}[\rm(a)]
\item Let $\mu$ be a volume density on $M$ with corresponding sub-Laplacian $\Delta_H$. Assume that $H$ is a proper subbundle in $TM$. Then there exists some connection $\nabla$ compatible with $g_H^*$ and satisfying $L(\nabla) f = \Delta_H f$. 
\item Let $g$ be a Riemannian metric taming $g_M$ and with volume form $\mu$. Let $\nabla$ be a connection compatible with both $g_H^*$ and $g$. Let $T^\nabla$ be the torsion of~$\nabla$ and define the one-from $\beta$ by
$$\beta(v) = \tr T^\nabla(v, \blank ).$$
Then the dual of $L = L(\nabla)$ on tensors is given by
$$L^* = L - 2\nabla_{\shh \beta} - \dv_\mu \shh \beta = L + (\nabla_{\shh \beta})^* - \nabla_{\shh \beta}.$$
In particular, $Lf = \Delta_{H} f + \langle \beta, df\rangle_{g_H^*}$ for any $f \in C^\infty(M)$.
\end{enumerate}
\end{lemma}

\begin{proof}
\begin{enumerate}[\rm (a)]
\item If $H$ is properly contained in $TM$, then there is some Riemannian metric $g$ such that $g|H = g_H$ and such that $\mu$ is the volume form $g$. Define $\nabla^0$ as in \eqref{Exists} and for any endomorphism valued one-form $\kappa \in \Gamma(T^*M \otimes \End T^*M)$, define a connection $\nabla^\kappa_v = \nabla^0_v + \kappa(v)$. The connection $\nabla^\kappa$ is compatible with $\nabla^0$ if and only if
\begin{equation} \label{KappaCompatible} \langle \kappa(v) \alpha, \alpha \rangle_{g_H^*} = 0, \quad v \in TM, \alpha \in T^*M.\end{equation}
Furthermore, $L(\nabla^\kappa) = L(\nabla^0) f +(\tr_H\kappa(\times)^\transpose \times )f$.

Define $Z = \Delta_H - L(\nabla^0)$. We want to show that there is an endo\-morphism-valued one-form $\kappa$ such that $\tr_H \kappa(\times)^\transpose \times = Z$ and such that \eqref{KappaCompatible} holds. By a partition of unity argument, it is sufficient to consider $Z$ as defined on a small enough neighborhood $U$ such that both $TM$ and $H$ are trivial. Let $\eta$ be any one-form on $U$ such that
$$\langle \eta , \beta \rangle_{g_H^*} =1, \quad \eta(Z) = 0.$$
Let $\zeta$ be a one-form such that $\shh \zeta = Z$. Define $\kappa$ by
$$\kappa(v) \alpha = \eta(v) \big( \alpha(Z) \eta-\alpha(\shh \eta) \zeta  \big).$$
This one-from $\kappa$ has the desired properties.

\item
For any connection $\nabla$ preserving the Riemannian metric $g$, we have
\begin{equation} \label{divergence} \dv_\mu Z = \sum_{i=1}^n \langle \nabla_{A_i} Z, A_i \rangle_{g} +  \sum_{s=1}^\nu \langle \nabla_{Z_s} Z, Z_s \rangle_{g} -  \beta(Z),\end{equation}
with respect to local orthonormal bases $A_1, \dots, A_n$ and $Z_1, \dots, Z_\nu$ of respectively $H$ and $V$.

For any pair of vector fields $A$ and $B$ consider an operator $F(A \otimes B) = \flat A \otimes \nabla_B $ on tensors with dual
$$F(A \otimes B)^* = -\iota_{(\dv B) A} - \iota_{\nabla_B A} -\iota_{A} \nabla_B.$$
Extend $F$ to arbitrary sections of $TM^{\otimes 2}$ by $C^\infty(M)$-linearity. Consider the operator $F(g_H^*)$. Since $\nabla$ preserves $H$, its orthogonal complement $V$ and their metrics, around any point $x$ we can find local orthonormal bases $A_1, \dots, A_n$ and $Z_1, \dots, Z_\nu$ of respectively $H$ and $V$ that are parallel at any arbitrary point~$x$. Hence, in any local orthonormal basis
$$ F(g_H^*)^* = \iota_{\shh \beta} - \sum_{i=1}^n \iota_{A_i} \nabla_{A_i}, $$
and so
\begin{equation*}
F(g_H^*)^* F(g_H^*) = - L + \nabla_{\shh \beta} = - L^* + \left(\nabla_{\shh \beta} \right)^* .\qedhere
\end{equation*}
\end{enumerate}
\end{proof}

\begin{remark}
As a result of the proof of Lemma~\ref{lemma:DualSL}, we actually know that all second order operators on the form $L(\nabla^0) +Z$ for some $Z \in \Gamma(H)$ is given as the rough sub-Laplacian of some connection compatible with the metric $g_H$.
\end{remark}

\section{Adjoint connections and infinite lifetime}

\subsection{A Weitzenb\"ock formula for sub-Laplacians} \label{sec:Weitz}

In the case of Riemannian geometry $g_H = g$, one of the central identities involving the rough Laplacian of the Levi-Civita connection $L(\nabla^g)$ is the Weitzenb\"ock formula $L(\nabla^{g})df  = \Ric_{g}(\sharp df, \blank) + dL(\nabla^{g}) f = \Ric_{g}(\sharp df, \blank) + d\Delta_{g} f$. A similar formula can be introduced in sub-Riemannian geometry, as was observed in \cite{EJL99} using the concept of adjoint connections. Adjoint connections was first considered in \cite{Dri92}.

If $\nabla$ is a connection on $TM$ with torsion $T^\nabla$, then its adjoint $\hat{\nabla}$ is defined by
$$\hat \nabla_A B = \nabla_A B- T^\nabla(A,B).$$
for any $A, B \in \Gamma(TM)$. We remark that $-T^\nabla$ is the torsion of $\hat{\nabla}$, so $\nabla$ is the adjoint of $\hat \nabla$.

\begin{proposition}[Sub-Riemannian Weitzenb\"ock formula] \label{prop:Weitz}
Let $L$ be any rough sub-Laplacian of an affine connection. Then there exists a vector bundle endomorphism $\mathscr{A}: T^*M \to T^*M$ such that for any $f \in C^\infty(M)$,
\begin{equation} \label{ZeroOrder} (L - \mathscr{A}) df = d Lf\end{equation}
if and only if $L = L(\hat{\nabla})$ for some adjoint $\hat{\nabla}$ of a connection $\nabla$ that is compatible with $g_H^*$. In this case, $\mathscr{A} = \Ric(\nabla)$, where
\begin{equation} \label{Ric} \Ric(\nabla)(\alpha)(v) := \tr_H R^\nabla(\times, v)\alpha (\times).\end{equation}
\end{proposition}
We note that bracket-generating assumption is not necessary for this result.
\begin{remark} \label{re:Conj}
\begin{enumerate}[\rm (i)]
\item Let $\nabla$ be a connection satisfying $\nabla g_H^* =0$ and let $\hat{\nabla}$ be its adjoint. By \cite[Proposition~2.1]{GoGr15} any smooth curve $\gamma$ in $M$ is a normal sub-Riemannian geodesic if and only if there is a one-form $\lambda(t)$ along $\gamma(t)$ such that
$$\shh \lambda(t) = \dot \gamma(t), \quad \text{ and } \quad \hat{\nabla}_{\dot \gamma} \lambda(t) = 0.$$
See the reference for the definition of normal geodesic. In this sense, adjoints of compatible connections occur naturally in sub-Riemannian geometry.
\item A Weitzenb\"ock formula in the sub-Riemannian case first appeared in \cite[Chapter~2.4]{EJL99}. See also \cite{Elw17}. This formulation assumes that the connection $\nabla$ can be represented as a Le Jan-Watanabe connection. For definition and the proof of the fact that all connections on a vector bundle compatible with some metric there are of this type, see \cite[Chapter~1]{EJL99}. We will do the proof of Proposition~\ref{prop:Weitz} without this assumption, in order to obtain an equivalence between existence of a Weitzenb\"ock formula and being an adjoint of a compatible connection.
 \end{enumerate}
\end{remark}

Before continuing with the proof, we will need the next lemma.
\begin{lemma} \label{lemma:Adjoiint}
Let $\nabla$ be an affine connection with adjoint $\hat{\nabla}$. Assume that~$\nabla$ is compatible with $g_H^*$ and denote  $L = L(\nabla)$, $\Ric = \Ric(\nabla)$ and $\hat{L} = L(\hat{\nabla})$.
For any endomorphism-valued one-form $\kappa \in \Gamma(T^*M \otimes \End T^*M)$ let $\nabla^\kappa$ be the connection
\begin{equation} \label{NablaKappa} \nabla^\kappa_v := \nabla_v + \kappa(v), \quad v \in TM. \end{equation}
\begin{enumerate}[\rm (a)]
\item If the horizontal bundle $H$ is a proper subbundle of $TM$, but bracket-generating, then the connection $\hat{\nabla}$ does not preserve $H$ under parallel transport.
\item \label{item:Lkappa} Define $L^\kappa = L(\nabla^\kappa)$. Then
$$L^\kappa = L + \nabla_{Z^\kappa}+ 2D^\kappa  + \kappa(Z^\kappa)+  \tr_H (\nabla_{\times}  \kappa)(\times) + \tr_H \kappa(\times)  \kappa(\times) $$
where $Z^\kappa = \tr_H \kappa(\times)^\transpose \times$ and $D^\kappa = \tr_H \kappa(\times) \nabla_{\times}$. In particular, for any function $f\in C^\infty(M)$,
$$L^\kappa f = L f+ Z^\kappa f \quad \text{ and } \quad \hat{L} f = Lf.$$
\item The adjoint $\hat{\nabla}^\kappa$ of $\nabla^\kappa$ is given by $\hat{\nabla}_v^\kappa = \hat{\nabla}_v + \hat{\kappa}(v)$ where
$$(\hat{\kappa}(v)\alpha)(w) := (\kappa(w) \alpha)(v), \quad \text{ for $v,w \in TM,\ \alpha \in T^*M.$}$$
In particular, if $\nabla^\kappa$ is compatible with $g_H^*$ then $\hat{\kappa}(\shh \alpha) \alpha = 0$ for any $\alpha \in T^*M$.
\end{enumerate}
\end{lemma}
\begin{proof}
\begin{enumerate}[\rm (a)]
\item Let $A, B \in \Gamma(H)$ be any two vector fields such that $[A,B]$ is not contained in~$H$.
Observe that $\hat{\nabla}_A B = \nabla_{B} A + [A,B]$ then cannot be contained in $H$ either.
\item This follows by direct computation: for any local orthonormal basis $A_1,\dots,A_n$ of $H$, we have
\begin{align*}
L^\kappa & = \sum_{i=1}^n \left(\nabla_{A_i} + \kappa(A_i) \right) \left(\nabla_{A_i} + \kappa(A_i) \right)\\
& \quad - \sum_{i=1}^n \left(\nabla_{\nabla_{A_i} A_i - \kappa(A_i)^\transpose A_i} + \kappa(\nabla_{A_i} A_i - \kappa(A_i)^\transpose A_i) \right) \\
& = \sum_{i=1}^n \nabla_{A_i} \nabla_{A_i}  + \sum_{i=1}^n \nabla_{A_i}  \kappa(A_i) + \sum_{i=1}^n  \kappa(A_i)  \nabla_{A_i} + \sum_{i=1}^n  \kappa(A_i)  \kappa(A_i)  \\
& \quad + \nabla_{Z^\kappa} + \kappa(Z^\kappa) - \sum_{i=1}^n \left(\nabla_{\nabla_{A_i} A_i } + \kappa(\nabla_{A_i} A_i) \right)     \\
&  = L  + 2 \tr_H  \kappa(\times)  \nabla_{\times} + \tr_H (\nabla_{\times}  \kappa)(\times)  + \tr_H \kappa(\times)  \kappa(\times) + \nabla_{Z^\kappa}^\kappa + \kappa(Z^\kappa)  .
\end{align*}
\item Follows from the definition and \eqref{KappaCompatible}
\qedhere
\end{enumerate}
\end{proof}

\begin{proof}[Proof of Proposition~\ref{prop:Weitz}]
Notice that $\iota_A \nabla_B df = \iota_{B} \hat{\nabla}_A df$. Since $\nabla$ is compatible with $g_H^*$, for any $x \in M$ there is a local orthonormal basis $A_1, \dots, A_n$ of $H$ such that $\nabla A_j (x) = 0$.
Hence, for an arbitrary vector field $Z \in \Gamma(TM)$, with the terms below evaluated at $x \in M$ implicitly,
\begin{align*}
\iota_Z dL(\hat \nabla)f & = \iota_Z d L(\nabla)f = Z \sum_{i=1}^n \nabla_{A_i} df(A_i) = \sum_{i=1}^n \nabla_Z \nabla_{A_i} df( A_i) \\
& = \sum_{i=1}^n \iota_{A_i}  R^{\nabla}(Z,A_i) df + \sum_{i=1}^n \nabla_{A_i} \nabla_Z df(A_i) + \nabla_{[Z,A_i]} df(A_i) \\
& = - \Ric(df)(Z) + \sum_{i=1}^n A_i \nabla_Z df(A_i) - \nabla_{\hat{\nabla}_{A_i} Z} df(A_i) \\
& = - \Ric(df)(Z) + \sum_{i=1}^n A_i \hat{\nabla}_{A_i} df(Z) - \hat{\nabla}_{A_i} df(\hat{\nabla}_{A_i} Z) \\
& = \iota_Z(- \Ric(df) +L(\hat \nabla) df).
\end{align*}
Since $x$ was arbitrary, it follows that $L(\hat{\nabla})$ satisfies \eqref{ZeroOrder}.

Conversely, suppose that $L= L(\nabla^\prime)$ is an arbitrary rough Laplacian of $\nabla^\prime$. Let~$\nabla$ be an arbitrary connection compatible with $g_H^*$ and define $\kappa$ such that $\nabla^\prime_{v} = \hat{\nabla}^\kappa_v = \hat{\nabla}_{v} + \hat{\kappa}(v)$, where $\nabla^\kappa$ is defined as in \eqref{NablaKappa}. Introduce vector field $Z =  \tr_H \hat{\kappa}(\times)^\transpose \times$ and first order operator $D =  \tr_H \hat{\kappa}(\times) \nabla_{\times}$. Using item~\eqref{item:Lkappa} of Lemma~\ref{lemma:Adjoiint}, modulo zero order operators applied to $df$, $L df - dLf$ equals $- dZf + \nabla_{Z} df + 2 D df$.
Furthermore, $- dZ f + \nabla_Z df = (\nabla_Z - \calL_Z) df$ and $(\nabla_Z - \calL_Z)$ is a zero order operator. Hence, it follows that \eqref{ZeroOrder} holds if and only if $D df = \mathscr{C} df$ for some zero order operator $\mathscr{C}$ and any $f \in C^\infty(M)$.

Let $A_1, \dots, A_n$ be a local orthonormal basis of $H$ and complete this basis to a full basis of $TM$ with vector fields $Z_1, \dots, Z_\nu$. Let $A_1^* ,\dots, A_n^*, Z_1^*, \dots, Z_\nu^*$ be the corresponding coframe.  Observe that $Z_1^* ,\dots, Z_\nu^*$ is a basis for $\Ann(H)$. For any $B \in \Gamma(TM)$ and $f \in C^\infty(M)$,
\begin{align*} (D df)(B) & = \sum_{i,k=1}^n\left( \hat{\kappa}(A_i) A_k^*(B) \right) \hat{\nabla}_{A_i} df(A_k) + \sum_{i=1}^n \sum_{s=1}^\nu \left( \hat{\kappa}(A_i) Z_s^*(B) \right) \hat{\nabla}_{A_i} df(Z_s).
\end{align*}
In order for this to correspond to a zero order operator, we must have $\hat{\kappa}(A_i) Z_s^* = 0$ and $\hat{\kappa}(A_i)(A_k^*) = - \hat{\kappa}(A_k)(A_i^*)$ which is equivalent to $\hat{\kappa}(\shh \alpha)\alpha = 0$ for any $\alpha \in T^*M$. Hence, $\hat{\nabla}^\kappa$ is the adjoint of a connection compatible with~$g_H^*$.
\end{proof}

\subsection{Connections with skew-symmetric torsion} \label{sec:IIzero}
For the sub-Riemannian manifold $(M, H, g_H)$ with $H$ strictly contained in $TM$, there exists no torsion-free connection which is compatible with the metric. Indeed, if $\nabla$ is a connection preserving~$H$, then the equality $\nabla_A B - \nabla_B A = [A,B]$ would imply that $H$ could be bracket-generating only if $H = TM$. For this reason, it has been difficult to find a direct analogue of the Levi-Civita connection in sub-Riemannian geometry.

For a Riemannian metric $g$, the only connections with the same geodesics as the Levi-Civita connection, are the compatible connections with \emph{skew-symmetric torsion}. These connections $\nabla$ compatible with $g$ such that
$$\zeta(v_1,v_2, v_3):= - \langle T^{\nabla}(v_1,v_2) , v_3 \rangle_{g}, \quad \text{$v_1$, $v_2$, $v_3$} \in TM,$$
is a well defined three-form. The connection $\nabla$ is then given by formula
$\nabla_A B = \nabla_A^g B - \frac{1}{2} \sharp \iota_{A \wedge B} \zeta$.
Equivalently, the connection $\nabla$ is compatible with $g$ and of skew-symmetric torsion if and only if we have both $\nabla g = 0$ and $\hat{\nabla} g = 0$. By Lemma~\ref{lemma:Adjoiint}~(a) we cannot have both $\nabla$ and $\hat{\nabla}$ being compatible with $g_H^*$ isn the proper sub-Riemannian case. In some cases, however, we have the following generalization.

Let $(M, H, g_H)$ be a sub-Riemannian manifold with taming Riemannian metric~$g$ and $V = H^\perp$. Let $\calL_A$ denote the Lie derivative with respect to the vector field~$A$. Introduce a vector-valued symmetric bilinear tensor $\II$ by the formula
\begin{equation} \label{II} \langle \II(A, A) , Z \rangle_{g} = - \frac{1}{2} (\calL_{\pr_{V} Z} g)(\pr_{H} A,\pr_{H} A) - \frac{1}{2} (\calL_{\pr_{H} Z} g)(\pr_{V} A,\pr_{V} A) \end{equation}
for any $A, Z \in \Gamma(TM)$. Observe that $\II = 0$ is equivalent to the assumption
\begin{equation} \label{IIzero} (\calL_A g) (Z, Z) = 0, \quad (\calL_Z g)(A, A) = 0, \end{equation}
for any $A \in \Gamma(H)$ and $Z \in \Gamma(V)$.
\begin{proposition} \label{prop:ChoiceCon}
Let $\nabla$ be a connection compatible with~$g_H^*$ and with adjoint~$\hat{\nabla}$. Assume that there exists a Riemannian metric~$g$ taming $g_H$ such that $\hat{\nabla} g = 0$. Then $\II=0$. Furthermore, if $\Delta_{H}$ is defined relative to the volume density of~$g$, then
$$\left(L(\hat \nabla) - \Ric(\nabla) \right) df = dL(\hat\nabla) f =  dL(\nabla)f = d\Delta_{H} f, \quad f \in C^\infty(M) .$$

Conversely,  suppose that $g$ is a Riemannian metric taming $g_H$ and satisfying $\II =0$.  Define $\calR$ and $\bar{\calR}$ as in \eqref{CurvCocurv} and introduce a three-form $\zeta$ by
\begin{equation} \label{TorsionForm}
\zeta(v_1,v_2, v_3) =  \circlearrowright \langle \calR(v_1, v_2), v_3 \rangle_{g} + \circlearrowright \langle \bar{\calR}(v_1, v_2), v_3 \rangle_{g}, \end{equation}
with $\circlearrowright$ denoting the cyclic sum. Then the connection
\begin{equation} \label{MaxRicCurv} \nabla_A B =  \nabla_A^{g} B - \frac{1}{2} \sharp \iota_{A \wedge B} \zeta \end{equation}
is a connection compatible with $g_H^*$ and both it and its adjoint $\hat{\nabla}_{A} B =\nabla_{A}^{g} B + \frac{1}{2} \sharp \iota_{A \wedge B} \zeta$ is compatible with $\hat{\nabla} g = 0$. 

Furthermore, among all such possible choices of connections, $\nabla$ gives the maximal value with regard to the lower bound of $\alpha \mapsto \langle \Ric(\nabla) \alpha, \alpha\rangle_{g_H^*}$.
\end{proposition}

\begin{remark} 
\begin{enumerate}[\rm (i)]
\item \emph{Analogy to the Levi-Civita connection:} Applying Proportion~\ref{prop:ChoiceCon} to the case when $g_H = g$ is a Riemannian metric, the Levi-Civita connection can be described as the connection such that both $\nabla$ and $\hat{\nabla}$ are compatible with $g$ and that also maximizes the lower bound $\alpha \mapsto \langle \Ric(\nabla) \alpha, \alpha \rangle_{g^*}$ whitch was observed in \cite[Corollary~C.7]{EJL99}. In this sense, the connection in \eqref{MaxRicCurv} is analogous to the Levi-Civita connection.
\item \emph{Existence and uniqueness for a Riemannian metrics $g$ taming $g_H$ and satisfying \eqref{IIzero}}: Every such Riemannian metric $g$ is uniquely determined by the orthogonal complement $V$ of $H$ and its value at one point \cite[Remark 3.10]{GrTh14a}. Conversely, suppose that $(M, H, g_H)$ is a sub-Riemannian manifold and let~$V$ be a subbundle such that $TM = H \oplus V$. Then one can use horizontal holonomy to determine if there exists a Riemannian metric $g$ taming $g_H$, satisfying \eqref{IIzero} and making~$H$ and~$V$ orthogonal. See \cite{CGJK15} for more details and examples where no such metric can be found. Two Riemannian metrics~$g_1$ and~$g_2$ may tame~$g_H$, satisfy \eqref{IIzero} and have the same volume density, but their orthogonal complements of $H$ may be different, see \cite[Example~4.6]{GrTh14a} and \cite[Example~4.2]{CGJK15}.
\item \emph{Geometric interpretation of \eqref{IIzero}}: From \cite{GoGr15}, the condition \eqref{IIzero} holds if and only if the Riemannian and the sub-Riemannian geodesic flow commute. See also Section~\ref{sec:FoliationCounter} for more relations to geometry and explanation of the notation~$\II$ for the tensor in \eqref{II}.
\item If we define $\nabla$ as in \eqref{MaxRicCurv} and assume $\bar{\calR} =0$, then its adjoint $\hat \nabla$ equals the connection $\nabla^\ve$ in \cite{Bau17} with $\ve = \frac{1}{2}$.
\end{enumerate}
\end{remark}

\begin{proof}
Let $\nabla^g$ be the Levi-Civita connection of $g$. Define the connection $\nabla^0$ as in \eqref{Exists} which is compatible with both $g_H^*$ and the Riemannian metric $g$. Let $T$ be the torsion of $\nabla^0$. Define $\calR$ and $\bar{\calR}$ as in \eqref{CurvCocurv}. Write $T_Z$ for the vector valued form $T_Z(A) = T(Z,A)$ and use similar notation for $\calR$, $\bar{\calR}$ and $\II$. By the definition of the Levi-Civita connection, we have
$$T_Z = - \calR_Z + \frac{1}{2} \calR_Z^* - \bar{\calR}_Z + \frac{1}{2} \bar{\calR}^*_Z + \II_Z^* - \II^*_{\blankdown} Z - \frac{1}{2} \calR^*_{\blankdown}Z - \frac{1}{2} \bar{\calR}_{\blankdown}^* Z,$$
with dual
$$T_Z^* = - \calR_Z^* + \frac{1}{2} \calR_Z - \bar{\calR}_Z^* + \frac{1}{2} \bar{\calR}_Z + \II_Z - \II^*_{\blankdown} Z + \frac{1}{2} \calR^*_{\blankdown}Z + \frac{1}{2} \bar{\calR}_{\blankdown}^* Z,$$
Hence, if we introduce $T^s_Z := \frac{1}{2} (T_Z + T^*_Z)$ then
$$2 T_Z^s = - \frac{1}{2} (\calR_Z + \calR^*_Z) - \frac{1}{2} (\bar{\calR}_Z + \bar{\calR}_Z^*) + (\II^*_Z + \II_Z) - 2 \II^*_{\blankdown} Z.  $$

Let $\nabla^\prime$ be a connection compatible with $g_H$. Define an $\End T^*M$-valued one-from $\kappa$ such that $\nabla_v^\prime = \nabla^\kappa_v = \nabla^0_v + \kappa(v)$, and let $\hat{\nabla}^\prime_v = \hat{\nabla}^0_v + \hat{\kappa}(v)$ be its adjoint. Define
$$\hat{\kappa}^s(Z) = \frac{1}{2} \left(\hat{\kappa}(Z) + \hat{\kappa}(Z)^*\right), \quad \hat{\kappa}^a(Z) = \frac{1}{2} \left(\hat{\kappa}(Z) - \hat{\kappa}(Z)^*\right).$$
In order for the adjoint to be compatible with $g$, we must have
$$(\hat{\nabla}^\kappa_Z g)(A,A) = 2 \langle (T_Z + \hat{\kappa}(Z)^\transpose) A, A \rangle_{g}=0,$$
giving us the requirement $\hat{\kappa}^s(Z)^\transpose = - T^s_Z$. However, since $\nabla^\kappa$ is compatible with $g_H$, we also have $\hat{\kappa}(\shh \alpha) \alpha = 0$ by Lemma~\ref{lemma:Adjoiint}. The latter condition is equivalent to $\hat{\kappa}(A)^{\transpose *} (A + B) = 0$ for any $A \in \Gamma(H)$ and $B \in \Gamma(V)$. This means that
\begin{align*} 
0 & = \langle \hat{\kappa}(A)^{\transpose *} (A + B), A + B \rangle_{g} = \langle \hat{\kappa}^s(A)^{\transpose} (A + B), A + B \rangle_{g} \\
& =- \langle T^s_A (A + B), A + B \rangle_{g}  = - \langle  \II(A ,A),  B \rangle_{g}  + \langle A, \II( B, B) \rangle_{g} .
\end{align*}
This holds for any $A \in \Gamma(H)$ and $B \in \Gamma(V)$ if and only if $\II =0$. It follows that $4\hat{\kappa}^s(Z)^\transpose = \calR_Z + \calR^*_Z + \bar{\calR}_Z + \bar{\calR}^*_Z$.

For the anti-symmetric part,
\begin{align*} 0 & = - 4 \hat{\kappa}(A)^{\transpose*}(A +B) = 4\hat{\kappa}^a(A)^\transpose (A+ B) - 4\hat{\kappa}^s(A)^\transpose(A +B)  \\
& =4\hat{\kappa}^a(A)^\transpose (A+ B) - \calR^*_A  B 
\end{align*}
for any $A \in \Gamma(H)$, $B \in \Gamma(V)$. This relation and anti-symmetry gives us
$$\hat{\kappa}^a(Z)^\transpose (A + B) = \hat{\kappa}^a(\pr_{V} Z)(A + B) -\frac{1}{4} ( \calR_Z- \calR_Z^*)(A+B) + \sharp \iota_{Z \wedge A} \beta,$$
where $\beta$ is a three-form vanishing on $V$.

In conclusion, for any $Z_1, Z_2 \in \Gamma(TM)$,
\begin{align*}  \nabla^\kappa_{Z_1} Z_2 &= \nabla_{Z_1}^0 Z_2 - \hat{\kappa}(Z_2)^\transpose (Z_1)  \\
& = \nabla_{Z_1}^0 Z_2 - \frac{1}{4} (  2 \calR^*_{Z_2} + \bar{\calR}_{Z_2} + \bar{\calR}^*_{Z_2}) Z_1 + \hat{\kappa}^a(\pr_{V} Z_2)(Z_1) + \sharp \iota_{Z_1 \wedge Z_2} \beta.
\end{align*}
Furthermore, since
\begin{align*} \nabla^0_Z & = \nabla^g_Z + \frac{1}{2} T_Z - \frac{1}{2}T_Z^* - \frac{1}{2} T^* Z \\
&= \nabla^g_Z + \frac{1}{2} \left(- \calR_Z + \frac{1}{2} \calR_Z^* - \bar{\calR}_Z + \frac{1}{2} \bar{\calR}^*_Z   - \frac{1}{2} \calR^*_{\blankdown}Z - \frac{1}{2} \bar{\calR}_{\blankdown}^* Z \right) \\
& \quad - \frac{1}{2}\left( - \calR_Z^* + \frac{1}{2} \calR_Z - \bar{\calR}_Z^* + \frac{1}{2} \bar{\calR}_Z  + \frac{1}{2} \calR^*_{\blankdown}Z + \frac{1}{2} \bar{\calR}_{\blankdown}^* Z\right) \\
&\quad  - \frac{1}{2} \left( - \calR_{\blankdown}^*Z - \frac{1}{2} \calR_Z - \bar{\calR}_{\blankdown}^*Z - \frac{1}{2} \bar{\calR}_Z  + \frac{1}{2} \calR^*_Z + \frac{1}{2} \bar{\calR}_Z^*\right) \\
& = \nabla^g_Z + \frac{1}{2} \left(-  \calR_Z + \calR_Z^* -  \bar{\calR}_Z + \bar{\calR}^*_Z   \right) ,
\end{align*}
we get
$$\nabla_{Z}^\kappa  = \nabla_Z^{g} + \frac{1}{2} \left(- \calR_Z + \calR_Z^* - \bar{\calR}_Z + \bar{\calR}_Z^* - \calR^* Z_1 - \bar{\calR}^* Z_1 \right)Z_2 + \lambda(Z_2) Z_1 + \shh \iota_{Z_1 \wedge Z_2} \beta$$
where $\lambda(Z) A = \frac{1}{4} (\bar{\calR}_Z - \bar{\calR}_Z^*) A- \hat{\kappa}^a(\pr_{V} Z)A$. 

In conclusion, if $\nabla^\prime$ and $\hat \nabla^\prime$ are compatible with $g_H^*$ and $g$ respectively, and $\nabla$ is defined as in \eqref{MaxRicCurv}, then $\II = 0$ and
\begin{equation} \label{GoodConnection} \nabla_{Z_1}^\prime Z_2 = \nabla^{\lambda,\beta}_{Z_1} Z_2 :=\nabla_{Z_1} Z_2 + \lambda(Z_2)Z_1 + \shh \iota_{Z_1 \wedge Z_2}\beta, \end{equation}
for some three-form $\beta$ vanishing on $V$ and some $\End TM$-valued one-form $\lambda$ vanishing on $H$ and satisfying $\lambda(v)^* = -\lambda(v)$, $v \in TM$. It is simple to verify that $\tr T^{\nabla^{\lambda,\beta}}(v, \, \cdot \,) = 0$ for any $v \in H$, and hence $L(\nabla^\prime) f = L(\hat \nabla^\prime) f = \Delta_H f$ by Lemma~\ref{lemma:DualSL}.

All that remains to be proven is that
$$\langle \alpha , \Ric( \nabla^{\lambda,\beta} ) \alpha \rangle_{g_H^*} \leq \langle \alpha , \Ric(\nabla) \alpha \rangle_{g_H^*}.$$
If $\nabla^\beta = \nabla^{0,\beta}$ then $\hat L^\beta : = L(\hat{\nabla}^{\beta})= L(\hat{\nabla}^{\lambda,\beta})$ since $\lambda$ vanishes on $H$. If we define $\hat{L} = L(\hat{\nabla})$, then for any smooth function~$f$ and local orthonormal basis $A_1, \dots, A_n$ of $H$,
\begin{align*}
& \hat L^\beta df(Z)  = \hat{L} df(Z)  + 2 \sum_{i=1}^n \hat{\nabla}_{A_i} df(\sharp \iota_{A_i \wedge Z}\beta) \\ & \qquad \qquad \qquad+ \sum_{i=1}^n df (\sharp \iota_{A_i \wedge Z} (\hat{\nabla}_{A_i}  \beta) ) + \sum_{i=1}^n df(\sharp_{A_i \wedge \sharp \iota_{A_i \wedge Z} \beta} \beta) \\
& = \hat{L}df(Z)  + \sum_{i=1}^n  df(T^{\nabla}(A_i ,\sharp \iota_{A_i \wedge Z}\beta)) + \sum_{i=1}^n (\hat{\nabla}_{A_i} \beta) (\sharp df ,A_i, Z) - 2 \langle \iota_{\sharp df} \beta, \iota_Z \beta \rangle_{\wedge^2 g_H^*} \\
& = \hat{L}df (Z)  + 2 \langle \iota_{\calR} df, \iota_{Z} \beta \rangle_{\wedge^2 g_H^*} - \tr_H (\hat{\nabla}_{\times} \beta) (\times, \sharp df , Z) - 2 \langle \iota_{\sharp df} \beta, \iota_Z \beta \rangle_{\wedge^2 g_H^*} .
\end{align*}
We use that
$$\langle (\hat L^\beta - \hat{L}) df, \alpha \rangle_{g} = \langle (\Ric(\nabla^\beta) -\Ric(\nabla)) df, \alpha \rangle_{g} = \langle (\Ric(\nabla^{\lambda,\beta}) -\Ric(\nabla)) df, \alpha \rangle_{g}.$$
As a consequence, for any $\alpha \in T^*M$,
\begin{align*}
& \langle \alpha, \Ric(\nabla^{\lambda, \beta}) \alpha \rangle_{g^*}  = \langle \alpha, \Ric(\nabla) \alpha\rangle_{g^*} +  2 \langle \iota_{\calR} \alpha, \iota_{\sharp \alpha} \beta \rangle_{\wedge^2 g_H^*} - 2 \langle \iota_{\sharp \alpha} \beta, \iota_{\sharp \alpha} \beta \rangle_{\wedge^2 g_H^*}. \end{align*}
Denoting $\alpha_H = \pr_{H}^* \alpha$, we get
\begin{align*}
 \langle \alpha, \Ric(\nabla^{\lambda, \beta}) \alpha \rangle_{g_H^*} & = \langle \alpha_H, \Ric(\nabla) \alpha_H \rangle_{g^*} - 2 | \iota_{\sharp \alpha_H} \beta|_{\wedge^2 g_H^*}^2. \end{align*}
The result follows.
\end{proof}

\subsection{Infinite lifetime of the sub-Riemannian Laplacian} \label{sec:InfiniteLifetime}
Assume now that the taming $g$ is a complete Riemannian metric. Then both the sub-Laplacian $\Delta_{H}$ of $\mu = \mu_g$ and the Laplacian $\Delta_{g}$ are essentially self-adjoint on compactly supported functions. We denote their unique self-adjoint extension by the same symbol.

Let $\nabla$ be a connection compatible with $g_H^*$ and let $X_t(\blank)$ be the stochastic flow of~$\frac{1}{2} L(\nabla)$ with explosion time $\tau(\blank)$. For any $x \in M$, let $\ptr_t = \ptr_t(x) : T_x M \to T_{X_t(x)} M$ be parallel transport along $X_t(x)$ with respect to $\nabla$. Using arguments similar to \cite[Section~2.5]{GrTh14a}, we know that the anti-development $W_t(x)$ at $x$ determined by
$$dW_t(x) = \ptr_t^{-1} \circ dX_t(x), \quad W_t(0) = 0 \in T_x M,$$
is a Brownian motion in the inner product space $(H_x, \langle \blank, \blank \rangle_{g_H(x)})$ with lifetime $\tau(x)$. Consider the semigroup $P_t$ on bounded Borel measurable functions corresponding to~$X_t(\blank)$
$$P_tf(x) = \mathbb{E}[1_{t < \tau(x)}f(X_t(x))].$$
We want to make statements regarding the explosion time $\tau(\blank)$ using connections that are compatible with $g_H^*$. Let $C^\infty_b(M)$ denote the space of smooth bounded functions. For vector bundle endomorphism $\mathscr{A}$ of $T^*M$ write $\mathscr{A}_{\ptr_t}(x) = \ptr^{-1}_t \mathscr{A}(X_t(x)) \ptr_t$ and let $\hptr_t$ denote the parallel transport along $X_t$ with respect to~$\hat \nabla$.

We make the following two assumptions
\begin{enumerate}[\rm (A)] 
\item \label{item:A} If $\II$ is defined as in \eqref{II}, then $\II = 0$.
\item \label{item:B} Consider the two-form $\calC\in \Gamma(\bigwedge^2 T^*M)$ defined by 
\begin{equation}
\label{calC} \calC(v,w) = \tr \bar{\calR}(v, \calR(w, \blank)) - \tr \bar{\calR}(w, \calR(v, \blank)), \quad v,w \in TM.
\end{equation}
Suppose that $\delta \calC = 0$ where $\delta$ is the codifferential with respect to $g$.
\item \label{item:C} Let $\nabla$ be defined as in \eqref{MaxRicCurv}. Assume that here exists a constant $K \geq 0$ such that for $\Ric = \Ric(\nabla)$,
$$\langle \Ric \alpha, \alpha \rangle_{g^*} \geq - K | \alpha|^2_{g^*} .$$
\end{enumerate}

\begin{theorem} \label{th:main}
Assume that \eqref{item:A}, \eqref{item:B} and \eqref{item:C} hold. Then we have the following results.
\begin{enumerate}[\rm (a)]
\item $\Delta_{g}$ and $\Delta_{H}$ spectrally commute. 
\item $\tau(x) = \infty$ a.s. for any $x \in M$.
\item Define
$\hat{Q}_t = \hat{Q}_t(x) \in \End T_x^*M$ as solution to the ordinary differential equation
$$\frac{d}{dt} \hat{Q}_t = - \frac{1}{2} \hat{Q}_t  \Ric_{\hptr_t} \,, \quad \hat{Q}_0 = \id.$$
Then, for any $f \in C^\infty_b(M)$ with $\| df\|_{L^\infty(g^*)} < \infty$, we have
$$dP_t f(x) = \mathbb{E}[ \hat{Q}_t \hptr_t^{-1} df(X_t(x))]$$
and $$\| dP_t f\|_{L^\infty(g^*)} \leq e^{Kt} \| d f\|_{L^\infty(g^*)}.$$ In particular,
$$\sup_{t \in [0, t_1]} \| dP_t f\|_{L^\infty(g^*)} \leq e^{Kt_1} \| df \|_{L^\infty(g^*)} < \infty$$ whenever $\| df\|_{L^\infty(g^*)} < \infty$.
\end{enumerate}
\end{theorem}
Remark that since $\nabla$ preserves $H$ under parallel transport, and hence also $\Ann(H)$, we have $\Ric \alpha = 0$ for any $\alpha \in \Ann(H)$. For this reason it is not possible to have a positive lower bound of $\langle \Ric \alpha, \alpha \rangle_{g^*}$ unless $H = TM$. The results of Theorem~\ref{th:main} appear as necessary conditions for the $\Gamma_2$-calculus on sub-Riemannian manifolds, see e.g. \cite{BaGa17,BKW16,GrTh14b} . We will use the remainder of this section to prove this statement.

\subsection{Anti-symmetric part of Ricci curvature}
Let $\zeta$ and $\nabla$ be as in \eqref{TorsionForm} and \eqref{MaxRicCurv}. The operator $\Ric(\nabla)$ is not in general symmetric. We consider the anti-symmetric part. Define $\Ric = \Ric(\nabla)$ and
\begin{equation} \label{RicAS} \Ric^s = \frac{1}{2} \left( \Ric + \Ric^*\right),\quad \Ric^a = \frac{1}{2} \left( \Ric - \Ric^*\right).\end{equation}

\begin{lemma} \label{lemma:AS}
For any $\alpha, \beta \in T^*M$,
$$2 \langle \Ric^a \alpha, \beta \rangle_{g^*} =  \tr_H (\nabla_{\times} \zeta)( \times, \sharp \alpha, \sharp \beta).$$
In particular,
$$ \langle \beta,  \Ric^a \alpha \rangle_{g^*} =\langle \pr_{V}^*\beta , \Ric^s \alpha \rangle_{g^*} - \langle \beta , \Ric^s \pr_{V}^* \alpha \rangle_{g^*},$$
so if $\Ric^s$ has a lower bound then $\Ric^a$ is a bounded operator.
Furthermore, if we define $\calC$ by \eqref{calC}, then whenever the $L^2$ inner product is finite
$$\langle \Ric^a df, d\phi \rangle_{L^2(g^*)} = \langle \calC, df \wedge d\phi \rangle_{L^2(g^*)} \quad \text{ for any $f,\phi \in C^\infty(M)$.}$$
\end{lemma}
The first part of this result is also found in \cite[Proposition C.6]{EJL99}. The condition $\Ric^a = 0$ is called \emph{the Yang-Mills condition} when $\bar{\calR} =0$. For more details, see Remark~\ref{re:geometric}.

\begin{proof}[Proof of Lemma~\ref{lemma:AS}]
Write $\circlearrowright$ for the cyclic sum. By the first Bianchi identity
\begin{align*}
& 2 \langle \Ric^a \alpha, \beta \rangle_{g^*} =  \sum_{i=1}^n \langle A_i, R^\nabla(A_i, \sharp \beta) \sharp \alpha - R^\nabla(A_i, \sharp \alpha) \sharp \beta \rangle_{g} \\
&= -  \sum_{i=1}^n \langle A_i, \circlearrowright R^\nabla(A_i , \sharp \alpha) \sharp \beta \rangle_{g} 
= - \sum_{i=1}^n \langle A_i ,\circlearrowright  (\nabla_{A_i} T)( \sharp \alpha, \sharp \beta) + \circlearrowright T( T(A_i, \sharp \alpha), \sharp \beta )  \rangle_{g} \\
&= \sum_{i=1}^n (\nabla_{A_i} \zeta)( A_i, \sharp \alpha, \sharp \beta)  - \sum_{i=1}^n \langle T(A_i, \sharp \alpha) , T(\sharp \beta, A_i) \rangle_{g}  - \sum_{i=1}^n \langle  T(\sharp \beta, A_i) , T( \sharp \alpha, A_i) \rangle_{g} \\
&= \tr_H (\nabla_{\times} \zeta)( \times, \sharp \alpha, \sharp \beta) . 
\end{align*}

Write $\zeta = \zeta_H + \zeta_V$, where $\zeta_H(v_1, v_2, v_3) = \circlearrowright \langle v_1, \calR(v_2, v_3) \rangle$ and $\zeta_V(v_1, v_2, v_3) = \circlearrowright \langle v_1, \bar{\calR}(v_2, v_3) \rangle$.   Since for $\alpha, \beta \in \Ann(H)$
$$2\langle \Ric^a \alpha, \beta \rangle_g = 0 = \tr_H (\nabla_\times \zeta)(\times, \sharp \alpha, \sharp \beta) =  \tr_H (\nabla_\times \zeta_V)(\times, \sharp \alpha, \sharp \beta),$$
we can write $\langle \Ric_a \alpha, \beta \rangle = \tr_H (\nabla_{\times} \zeta_H)( \times, \sharp \alpha, \sharp \beta)$. Observe also that by reversing the place of $V$ and $H$ and writing $g_V = g|V$, we have also $\tr_{g_V} (\nabla_\times \zeta_H)(\times, \sharp \alpha, \sharp \beta) =0$.

Continuing, if $A_1, \dots, A_n$ and $Z_1, \dots, Z_\nu$ are local orthonormal bases of $H$ and~$V$, respectively, observe that since $\nabla$ preserves the metric $g$, for any one-from $\eta$, we have
$$d\eta = \sum_{i=1}^n \flat A_i \wedge \nabla_{A_i} \eta + \sum_{i=1}^\nu \flat Z_\nu \wedge \nabla_{Z_\nu} \eta  + \iota_{T} \eta,$$
where $\iota_T \eta = \eta(T(\, \cdot \, , \, \cdot \,))$. The formula above becomes valid for arbitrary forms $\eta$ if we extend $\iota_T$ by rule $\iota_T (\alpha \wedge \beta) = (\iota_T \alpha) \wedge \beta + (-1)^k \alpha \wedge \iota_T \beta$ for any $k$-form $\alpha$ and form $\beta$. Observe that for any $v \in TM$, $\tr T(v, \, \cdot \,) = 0$. Using similar arguments to the proof of Lemma~\ref{lemma:DualSL}~(b), we obtain a local formula for the codifferential
$$\delta \eta=  - \sum_{i=1}^n \iota_{ A_i}\nabla_{A_i} \eta - \sum_{i=1}^\nu \iota_{Z_\nu}  \nabla_{Z_\nu} \eta  + \iota_{T}^* \eta.$$
This means that
\begin{align*} &  \tr_H (\nabla_{\times} \zeta_H)( \times, \sharp \alpha, \sharp \beta) = (\iota_T^* \zeta_H)(\sharp \alpha, \sharp \beta) - (\delta \zeta_H)( \sharp \alpha, \sharp \beta) \\
& = \langle \calC - \delta \zeta_H, \alpha \wedge \beta\rangle_{g^*}
\end{align*}
Inserting $\alpha \wedge \beta = df \wedge d\phi = d(f d\phi)$ and integrating over the manifold, we have the result.
\end{proof}

\subsection{Commutation relations with between the Laplacian and the sub-Laplacian}
Let $(M, H, g_H)$ be a sub-Riemannian manifold and let $g$ be a taming Riemannian metric with $\II = 0$. Define $\Delta_g$ as the Laplacian of $g$ and let $\Delta_H$ be defined relative to the volume form of $g$.
\begin{proposition} \label{prop:Commute} 
Define $\calC$ as in \eqref{calC}.
\begin{enumerate}[(a)]
\item We have $\Delta_{g} \Delta_{H} f = \Delta_{H} \Delta_{g} f$, for all $f\in C^\infty(M)$ if and only if $\delta \calC = 0$.
\item Assume $\delta \calC = 0$ and that $\Ric(\nabla)$ is bounded from below by a constant $-K$. Then $\Delta_{g}$ and $\Delta_{H}$ spectrally commute.
\end{enumerate}
\end{proposition}

Before stating the proof, we will need the following lemmas.
\begin{lemma}[\rm{\cite[Proposition]{Li84}}, {\cite[Proposition 4.1]{BaGa17}}]  \label{lemma:Lpunique} 
Let $\mathcal{A}$ be equal to the Laplacian $\Delta_{g}$ or sub-Laplacian $\Delta_{H}$ defined relative to a complete Riemannian or sub-Riemannian metric, respectively. Let $M \times [0, \infty)$, $(x,t) \mapsto u_t(x)$ be a function in $L^2$ of the solving the heat equation
$$(\partial_t - \mathcal{A}) u_t = 0, \quad u_0 = f,$$
for an $L^2$-function $f$. Then $u_t(x)$ is the unique solution to this equation in $L^2$.
\end{lemma}

\begin{lemma} \label{lemma:DualHatSL}
Let $(M, H, g_H)$ be a sub-Riemannian manifold and define $\Delta_H$ as the sub-Laplacian with respect to a volume form $\mu$. Let $g$ be a taming metric of $g_H$ with volume form $\mu$. Assume that $\nabla$ and its adjoint $\hat \nabla$ are compatible with $g_H^*$ and $g$, respectively.
If $\hat{L} = L(\hat{\nabla})$, then with respect to $g$, $\hat{L}^* = \hat{L} = - (\hat \nabla_{\pr_H})^* \hat\nabla_{\pr_H}$. In particular, $\hat{L}f = \Delta_{H} f$ for any $f \in C^\infty(M)$.
\end{lemma}
\begin{proof}

Define $\hat{F}(A\otimes B) = \flat A \otimes \hat{\nabla}_B$ and extend it by linearity to all sections of $TM^{\otimes 2}$. Again we know that for any point $x$, there exists a basis $A_1, \dots, A_n$ such that $\nabla A_i(x) = 0$. This means that $\hat{\nabla}_Z A_i(x) = T^{\nabla}(A_i, Z)(x)$ for the same basis, and hence locally
$$\hat{F}(g_H^*)^* = - \iota_{\shh \hat{\beta}} - \sum_{i=1}^n \iota_{A_i} \hat{\nabla}_{A_i}, \quad \hat{\beta}(v) = \tr T^{\hat{\nabla}}(v, \blank).$$
However, since $\hat{\nabla}$ is the adjoint of a connection compatible with $g_H^*$ we have $\hat{\beta}  = 0$ since $\hat{\nabla}$ has to be on the form \eqref{GoodConnection}. Hence $\hat{F}(g_H^*)^* \hat{F}(g_H^*) = - \hat L$ and the result follows.
\end{proof}

\begin{proof}[Proof of the Proposition~\ref{prop:Commute}]\ 
\begin{enumerate}[\rm (a)] 
\item It is sufficient to prove the statement for compactly supported functions. 
Note that for $f,\phi \in C^\infty_c(M)$, $\langle \Delta_{H} \Delta_{g} f, \phi \rangle_{L^2} = \langle f, \Delta_{g} \Delta_{H} \phi \rangle_{L^2}$.
Hence, we need to show that $\Delta_{g} \Delta_{H}$ is its own dual on compact supported forms.

Let $\nabla$ be as in \eqref{MaxRicCurv} with adjoint $\hat{\nabla}$. Define $L = L(\nabla)$, $\hat{L} = L(\hat{\nabla})$, $\Ric = \Ric(\nabla)$ and introduce $\Ric^a = \frac{1}{2} \left( \Ric - \Ric^* \right).$
By Lemma~\ref{lemma:DualHatSL} we have $\hat{L}^* = \hat{L}$. In addition,
\begin{align*}
& \langle \Delta_{g} \Delta_{H} f, \phi \rangle_{L^2(g^*)} = - \langle dL f, d\phi \rangle_{L^2(g^*)}
= - \langle (\hat{L} - \Ric) df, d\phi \rangle_{L^2(g^*)} \\
& =- \langle df, (\hat{L} - \Ric) d\phi \rangle_{L^2(g^*)} + 2 \langle \Ric^a df, d\phi \rangle_{L^2(g^*)} \\
&  =\langle f, \Delta_{g} \Delta_{H} \phi \rangle_{L^2} + 2 \langle \Ric^a df, d\phi \rangle_{L^2(g^*)} .
\end{align*}
Furthermore, $2 \langle \Ric^a df, d\phi \rangle_{L^2} = \langle \calC , df \wedge d\phi \rangle_{L^2} = \langle \delta \calC, f d\phi \rangle_{L^2}$. Since all one-forms can we written as sums of one-forms of the type $f d\phi$, it follows that $(\Delta_{g} \Delta_{H})^*f = \Delta_{g} \Delta_{H}f$ for $f \in C_c^\infty(M)$ if and only if $\delta \calC = 0$.
\item Write $\Delta_g = \Delta_H + \Delta_V$ and $df = d_H f + d_Vf$, with $d_Hf = \pr_H^* df$ and $d_Vf = \pr_V^* df$. Then $\langle \Delta_H f, \phi \rangle_{L^2(g^*)} = -\langle d_H f, d_H \phi \rangle$ and similarly for $\Delta_V$.

Observe that for any compactly supported $f$,
\begin{align*}
& \| \Delta_g f \|_{L^2} \| \Delta_H f\|_{L^2} \geq \langle \Delta_g f, \Delta_H f\rangle_{L^2} =- \langle df, (\hat L - \Ric) df \rangle_{L^2(g^*)} \\
&= \| \hat{\nabla} df \|^2_{L^2(g^{*\otimes 2})} + \langle df, \Ric d_H f\rangle_{L^2(g^*)}  \geq \frac{1}{n} \| \Delta_Hf \|^2_{L^2} - K \|df\|_{L^2(g^*)} \| d_H f \|_{L^2(g^*)}.
\end{align*}
and ultimately
\begin{equation} \label{DeltaHBoundedByDelta}
 \| \Delta_Hf \|^2_{L^2}  \leq n  \sqrt{\| \Delta_g f \|_{L^2} \| \Delta_H f\|_{L^2}}(\sqrt{\| \Delta_g f \|_{L^2} \| \Delta_H f\|_{L^2}} + K \| f\|_{L^2}).
\end{equation}
By approaching any $f \in \Dom(\Delta_g)$ by compactly supported functions, we conclude from \eqref{DeltaHBoundedByDelta} that any such function must satisfy $\| \Delta_H f\|_{L^2} < \infty$. As a consequence, $\Dom(\Delta_g) \subseteq \Dom(\Delta_{H})$.

Let $Q_ t= e^{t \Delta_{g} /2}$ and $P_t = e^{t \Delta_{H} /2}$  be the semigroups of $\Delta_{g}$ and $\Delta_{H}$, which exists by the spectral theorem. For any
$f \in \Dom(\Delta_{H})$, $u_t = \Delta_{H} Q_t f$ is an $L^2$ solution of
$$\left(\frac{\partial}{\partial t} - \frac{1}{2} \Delta_g\right) u_t = 0, \quad u_0 =
\Delta_{H} f.$$ By Lemma~\ref{lemma:Lpunique} we obtain $\Delta_{H}
Q_t f = Q_t \Delta_{H} f$. Furthermore, for any $s >0$ and $f \in L^2$, we know that $Q_sf \in
\Dom(\Delta_g) \subseteq \Dom(\Delta_{H})$, and since
$$\left(\frac{\partial}{\partial t} -
\frac{1}{2} \Delta_{H} \right) Q_s P_t f = 0,$$ it again follows from
Lemma~\ref{lemma:Lpunique} that $P_t Q_s f = Q_s P_t f$ for any $s,t
\geq 0$ and $f \in L^2$. The operators
consequently spectrally commute, see \cite[Chapter VIII.5]{ReSi80}.
\end{enumerate}
\end{proof}

\begin{remark} The results of Lemma~\ref{lemma:DualSL} and Lemma~\ref{lemma:DualHatSL} do not need the bracket generating assumptions. The result of $\hat L$ being symmetric is also found in \cite[Theorem~2.5.1]{EJL99} for the case when $\nabla$ and $\hat \nabla$ preserves the metric.
\end{remark}

\subsection{Proof of Theorem~\ref{th:main}} \label{sec:Lifetime}
We consider the assumptions that $\delta \calC = 0$ and that the symmetric part $\Ric^s$ of the $\Ric$ is bounded from below. By Lemma~\ref{lemma:AS}, the anti-symmetric part $\Ric^a$ is a bounded operator. Furthermore, the operators $\Delta_{g}$ and $\Delta_{H}$ spectrally commute by Proposition~\ref{prop:Commute}.

Let $X_t(x)$, $\hptr_t$ and $\hat{Q}_t$ be as in the statement of the theorem. If $$N_t = \hat{Q}_t \hptr_t^{-1} \alpha(X_t(x))$$ for an arbitrary $\alpha \in \Gamma(T^*M)$, then by It\^o's formula
$$d  N_t  \stackrel{\text{loc.\,m.}}{=} \hat{Q}_t \hptr_t^{-1} (\hat{L} - \Ric)\alpha(X_t(x)) dt$$
where $\stackrel{\text{loc.\,m.}}{=} $ denotes equivalence modulo differential of local martingales. Consider $L^2(T^*M)$ as the space of $L^2$-one-forms on $M$ with respect to $g$. Since $g$ is complete and $\Ric^s$ bounded from below, the operator $\hat{L} - \Ric^s$ is essentially self-adjoint by Lemma~\ref{lemma:DualHatSL} and Lemma~\ref{lemma:SA}. Hence, by Lemma~\ref{lemma:StochRep}, there is a strongly continuous semigroup $P^{(1)}_t$ on $L^2(TM)$ with generator $(\hat{L} - \Ric, \Dom(\hat{L} - \Ric^s))$ such that
$$P^{(1)}_t\alpha(x) = \mathbb{E}[1_{t < \tau(x)} N_t]= \mathbb{E}[1_{t < \tau(x)}\hat{Q}_t \hptr_t^{-1} \alpha(X_t(x))].$$

We want to show that for any compactly supported function $f$, $P^{(1)}_t df = dP_t f$ where $P_t f(x) = \mathbb{E}[f(X_t(x)) 1_{t< \tau(x)}]$. Following the arguments in \cite[Appendix~B.1]{DrTh01}, we have $P_t f = e^{t\Delta_{H}/2}f$ where the latter semigroup is the $L^2$-semigroup defined by the spectral theorem and the fact that $\Delta_{H}$ is essentially self-adjoint on compactly supported functions. To this end, we want to show that $dP_t f$ is contained in the domain of the generator of $P^{(1)}_t$. This observation will then imply that $P_t^{(1)} df = dP_t f$, since $P_t^{(1)} df$ is the unique solution to
$$\frac{\partial}{\partial t} \alpha_t = \frac{1}{2} L \alpha_t, \quad \alpha_0 = df,$$
with values in $\Dom(\hat{L}- \Ric^s)$ by strong continuity, \cite[Chapter II.6]{EnNa00}.

We will first need to show that $dP_t f$ is indeed in $L^2$. Let $ \Delta_{g}$ denote the Laplace-Beltrami operator of $g$, which will also be essentially self-adjoint on compactly supported functions since $g$ is complete. Denote its unique self-adjoint extension by the same symbol. Since the operators spectrally commute, $e^{s\Delta_{g}} e^{t\Delta_{H}} = e^{t\Delta_{H}} e^{s\Delta_{g}}$  for any $s,t \geq 0$ which implies $\Delta_{g} e^{t \Delta_{H}} f= e^{t \Delta_{H}} \Delta_{g} f$ for any~$f$ in the domain of $\Delta_{g}$. In particular,
$$\langle d P_t f, dP_tf \rangle_{L^2(g^*)} = - \langle \Delta_{g} P_t f, P_t f\rangle_{L^2(g^*)} = - \langle P_t \Delta_{g} f, P_t f \rangle_{L^2(g^*)} < \infty.$$

Next, since $\langle (\hat{L} - \Ric^s) \alpha, \alpha \rangle_{L^2(g^*)} \geq - K \| \alpha \|_{L^2(g^*)}^2$, the domain $\Dom(\hat{L} - \Ric^s)$ equals the completion of compactly supported one-forms $\Gamma_c(T^*M)$ with respect to the quadratic form
\begin{align*} q(\alpha, \alpha) & = (K+1)\langle \alpha, \alpha \rangle_{L^2(g^*)} - \langle (\hat{L} -\Ric^s) \alpha, \alpha \rangle_{L^2(g^*)} \\
& = (K+1)\langle \alpha, \alpha \rangle_{L^2(g)} - \langle (\hat{L} -\Ric) \alpha, \alpha \rangle_{L^2(g^*)}.\end{align*}
Since $P_tf$ is in the domain of both $\Delta_{g}$ and $\Delta_{H}$ for any compactly supported $f$, we have that for any fixed $t$, there is a sequence of compactly supported functions $h_n$ such that $h_n \to P_t f$, $\Delta_{H} h_n \to \Delta_{H} P_t f$ and $\Delta_{g} h_n \to \Delta_{g} P_t f$ in $L^2$. From the latter fact, it follows that $dh_n$ converges to $dP_tf$ in $L^2$ as well. Furthermore,
\begin{align*}
q(dh_n, dh_n) & = (K+1) \langle d h_n, dh_n \rangle_{L^2(g)} - \langle (\hat{L} - \Ric) dh_n, dh_n \rangle_{L^2(g)} \\
& = - (K+1) \langle h_n, \Delta_{g} h_n \rangle_{L^2(g)} - \langle d \Delta_{H} h_n , dh_n \rangle_{L^2(g)} \\
& = - (K+1) \langle h_n, \Delta_g h_n \rangle_{L^2(g)} + \langle  \Delta_{H} h_n , \Delta_g h_n \rangle_{L^2(g)},
\end{align*}
which has a finite limit as $n \to \infty$. Hence, $dP_t f \in \Dom(\hat{L} - \Ric^s)$ and $P^{(1)}_t df  = dP_tf$.

Using that $\langle \Ric \alpha , \alpha\rangle_{g^*} \geq - K | \alpha|_{g^*}^2$, Gronwall's lemma and the fact that $\hat{\nabla}$ preserves the metric means that
$$|1_{t < \tau(x)} \hat{Q}_t \hptr_t^{-1} \alpha(X_t(x)) |_{g^*} \leq e^{Kt/2} 1_{t< \tau(x)} | \alpha |_{g^*}(X_t(x)).$$
Hence,
\begin{equation} \label{ineq} | P_t^{(1)} \alpha (x) |_{g^*} \leq e^{Kt/2} P_t |\alpha|_{g^*} (x).\end{equation}
We assumed that $g$ was complete, so we know that there exists a sequence of compactly supported functions $g_n$ such that $g_n \uparrow 1$ and such that $\|dg_n\|^2_{L^\infty(g^*)} \to 0$. Since $| dP_t g_n |_{g^*} \to 0$ uniformly by \eqref{ineq} and we know that $P_t g_n \to P_t1$, we obtain $dP_t 1 = 0$. Hence, we know that $P_t 1 =1$, which is equivalent to $\tau(x) = \infty$ almost surely.

It is a standard argument to extend the formulas from functions of compact support to bounded functions with $\| df \|_{L^\infty(g^*)} < \infty$.

\subsection{Foliations and a counter-example} \label{sec:FoliationCounter}
Let $(M, H, g_H)$ be a sub-Riemannian manifold and let $g$ be a Riemannian metric taming $g_H$ and satisfying $\II = 0$ with $\II$ as in \eqref{II}. Write $V$ for the orthogonal complement of $H$. Define \emph{the Bott connection}, by
\begin{align} \label{rnabla} \rnabla_{Z_1} Z_2 & = \pr_{H} \nabla^{g}_{\pr_{H} Z_1} \pr_{H} Z_2 + \pr_{V} \nabla^{g}_{\pr_{V} Z_1} \pr_{V} Z_2 \\ \nonumber
& \qquad + \pr_{H} [\pr_{V} Z_1, \pr_{H} Z_2] + \pr_{V} [\pr_{H} Z_1, \pr_{V} Z_2]
\end{align}
where $\nabla^{g}$ denote the Levi-Civita connection. Its torsion $\mathring{T} := T^{\rnabla}$ equals $\mathring{T} = - \calR - \bar{\calR}$ and $\rnabla g = 0$ is equivalent to requiring $\II = 0$. Since $\rnabla$ is compatible with the metric, we have
$$\rnabla_Z = \nabla^{g}_Z + \frac{1}{2} \mathring{T}_Z - \frac{1}{2} \mathring{T}^*_Z - \frac{1}{2} \mathring{T}^*_{\blankdown} Z, \quad T_Z(A) = T(Z,A).$$
If $\zeta$ and $\nabla$ are as in \eqref{TorsionForm} and \eqref{MaxRicCurv}, respectively, then
$$\zeta(v_1, v_2, v_2) = -\circlearrowright \langle \mathring{T}(v_1, v_2), v_3 \rangle_{g}, \quad \text{and} \quad \nabla_{Z}  = \rnabla_{Z} + \mathring{T}_{\blankdown}^* Z.$$
The connection $\rnabla$ does not have skew-symmetric torsion, however, it does have the advantage that $\rnabla_A B$ is independent of $g|{V}$ if either $A$ or $B$ takes its values in $H$, see \cite[Section~3.1]{GrTh14a}.

\subsubsection{Totally geodesic, Riemannian foliations} \label{sec:Foliations} Assume that $\bar{\calR} = 0$, i.e.~assume that the orthogonal complement $V$ of $H$ is integrable. Let $\calF$ be the corresponding foliation of $V$ that exists from the Frobenius theorem. We have the following way of interpreting the condition $\II = 0$. The tensor $\II(\pr_{V} \blank, \pr_{V} \blank)$ equals the second fundamental form of the leaves, i.e.~$\pr_{H} \nabla^{g}_Z W = \II(Z,W)$ for any $Z, W \in \Gamma(V)$. Hence, $\II(\pr_{V} \blank, \pr_{V} \blank) = 0$ is equivalent to the leaves of~$\calF$ being totally geodesic immersed submanifolds. On the other hand, the condition $0 = -2 \langle \II(A, A ), Z \rangle = (\calL_Z g)(A,A)$ for any $A \in \Gamma(H)$, $Z \in \Gamma(V)$ is the definition of $\calF$ being \emph{a Riemannian foliations}. Locally, such a foliation $\calF$ consists of the fibers of a Riemannian submersion. In other words, every $x_0 \in M$ has a neighborhood $U$ such that there exists a surjective submersion between two Riemannian manifolds,
\begin{equation} \label{Rsubmersion} \pi:(U, g|_U) \to ( \check{M}_U, \check{g}_U),\end{equation}
satisfying $$TU = H|U \oplus_{\perp} \ker \pi_*, \quad\calF|U = \{ \pi^{-1}(\check{x}) \colon\check{x} \in \check{M}_U \}$$ and that $\pi_*\colon H_x \to T_{\pi(x)} \check{M}_U$ is an isometry for every $x \in U$.

The following result is found in \cite{Elw14} for totally geodesic Riemannian foliations. Let $X_t(\blank)$ be a stochastic flow with generator $\frac{1}{2}\Delta_{H}$ where the latter is define relative to the volume density of $g$.
\begin{theorem}
If $(M, g)$ is a stochastically complete Riemannian manifold, then $X_t(x)$ has infinite lifetime.
\end{theorem}
In particular, if the Riemannian Ricci curvature $\Ric_{g}$ is bounded from below, $X_t(x)$ has infinite lifetime. We want to compare this result using the entire Riemannian geometry with our result using $\Ric(\nabla)$, an operator only defined by taking the trace over horizontal vectors. For this special case, it turns out that $\Ric_{g}$ being bounded from below is actually a weaker condition than $\Ric(\nabla)$ being bounded from below.

\begin{proposition} \label{prop:LocSym}
Let $(M, H, g_H)$ be a sub-Riemannian manifold with $H$ is bracket-generating. Let $\calF$ be a foliation of $M$ corresponding to an integrable subbundle $V$ such that $TM = H \oplus V$. Let $g$ be any Riemannian metric taming $g_H$ such that $\II=0$, making $\calF$ a totally geodesic Riemannnian foliation. Assume finally that $g$ is complete.
For $x \in M$, let $F_x$ denote the leaf of the foliation $\calF$ containing $x$. Write $\Ric_{F_x}$ for the Ricci curvature tensor of $F_x$.
\begin{enumerate}[\rm (a)]
\item For any $x, y \in M$, there exist neighborhoods $x \in U_x \subseteq F_x$ and $y \in U_y \subseteq F_y$, and an isometry
$$\Phi:U_x \to U_y, \quad \Phi(x) = y.$$
As a consequence, if we define $\Ric_{\calF}$ such that
$$\Ric_{\calF}(v,w) = \Ric_{F_x} (\pr_{V}v, \pr_{V} w), \quad \text{ for any } v, w \in T_x M,$$
then $\Ric_{\calF}$ is bounded.
\item Let $\Ric_{g}$ be the Ricci curvature of the Riemannian metric $g$. Let $\nabla$ be defined as in \eqref{MaxRicCurv}. Then
\begin{align} \label{RicGRic}
\quad \Ric_{g}(v,v) &= \Ric(\nabla)(\flat v)(v)  + \frac{1}{2} \sum_{i=1}^n \| \calR(A_i, v) \|_{g}^2  + \Ric_{\calF}(\pr_{V} v, \pr_{V} v).
\end{align}
In particular, $\Ric_{g}$ has a lower bound if $\Ric(\nabla)$ has a lower bound.
\end{enumerate}
\end{proposition}

Before presenting the proof we need the next lemma. Let $(M, g)$ be a complete Riemannian manifold and let~$\calF$ be a Riemannian foliation with totally geodesic leaves. Let $V$ be the integrable subbundle of $TM$ corresponding to $\calF$ and define~$H$ as its orthogonal complement. Write $n$ for the rank of $H$ and $\nu$ for the rank of~$V$. Define
$$\Ort(n) \to \Ort(H) \stackrel{p}{\to} M$$
as the orthonormal frame bundle of $H$. Introduce the principal connection $E$ on~$p$ corresponding to the restriction of $\rnabla$ to $H$. In other words, $E$ is the subbundle of $T\Ort(H)$ satisfying $T\Ort(H) = E \oplus \ker p_*$, $E_{\phi} \cdot a= E_{\phi \cdot a}$, $\phi \in \Ort(H)$, $a \in \Ort(n)$ and defined such that a curve $\phi(t)$ in $\Ort(H)$ is tangent to $E$ if and only if the frame is $\rnabla$-parallel along $p(\phi(t))$. For any $u = (u_1, \dots, u_n) \in \mathbb{R}^n$, define $\hat{A}_u$ as the vector field on $\Ort(H)$ taking values in $E$ uniquely determined by the property
$$p_* \hat{A}_u(\phi) = \sum_{j=1}^n u_j \phi_j , \quad \text{ for any } \phi = (\phi_1, \dots, \phi_n) \in \Ort(H).$$

For any $\phi \in \Ort(H)_x$, define $\hat{F}_{\phi}$ as all points that can be reached from $\phi$ by an $E$-horizontal lift of a curve in $F_x$ starting in $x$. We then have the following result, found in \cite{Elw14}, see also \cite[Chapter~10]{Ton97} and \cite{Mol73}. 
\begin{lemma}
The collection $\hat{\calF} = \{ \hat{F}_\phi\colon \phi \in \Ort(H)\}$ gives a foliation of $\Ort(H)$ with $\nu$-dimensional leaves such that for each $\phi \in O(n)$ the map $$p|{\hat{F}_\phi} \colon \hat{F}_\phi \to F_{p(\phi)}$$ is a cover map. Furthermore, giving each leaf of $\hat{\calF}$ a Riemannian structure by pulling back the metric from the leaves of $\calF$, then for any $u \in \mathbb{R}^n$ and $t \in  \mathbb{R}$, the flow $\Psi_u(t) = e^{t\hat{A}_u}$ maps $\hat{F}_\phi$ onto $\hat{F}_{\Psi_u(t)(\phi)}$ isometrically for each $\phi \in \Ort(H)$.
\end{lemma}

Note that the reason for using the connection $\rnabla$ in the definition of $\hat{\calF}$, is that $R^{\rnabla}(Z,W) A = 0$ for any $Z, W \in \Gamma(V)$ and $A \in \Gamma(H)$.

\begin{proof}[Proof of Proposition~\textup{\ref{prop:LocSym}}]
\begin{enumerate}[\rm (a)]
\item
For any $x \in M$, choose a fixed element $\phi_0$ in $\Ort(H)_x$. Define
$$\mathcal{O}_{\phi_0} = \left\{ \Psi_{u_k}(t_k) \circ  \cdots \circ \Psi_{u_1}(t_1)(\phi) : \, t_j \in \mathbb{R}, u_j \in \mathbb{R}^n,\ k \in \mathbb{N} \right\}.$$
Clearly, by definition, for any $\phi \in \mathcal{O}_{\phi_0}$, there is an isometry $\hat{\Phi}: \hat{F}_{\phi_0} \to \hat{F}_{\phi}$ such that $\hat{\Phi}(\phi_0) = \phi$.

Consider the vector bundle $\hat H = \spn \{ \hat A_u \colon \, u \in \mathbb{R}^n \}$ and define  
\begin{align*}
\Lie_{\phi} \hat H & := \spn \left\{ [B_1, [B_2, \cdots, [B_{k-1}, B_k]] \cdots] \big|_{\phi} : \, B_j \in \Gamma(\hat H),\ k \in \mathbb{R} \right\} \\
& \, = \spn \left\{ [\hat{A}_{u_1}, [\hat{A}_{u_2}, \cdots, [\hat{A}_{u_{k-1}}, \hat{A}_{u_k}]] \cdots] \big|_{\phi}: \, u_j \in \mathbb{R}^n,\ k \in \mathbb{R} \right\},
\end{align*}
for any $\phi \in \Ort(H)$.
By the Orbit Theorem, see e.g.~\cite[Chapter~5]{AgSa04}, $\mathcal{O}_{\phi_0}$ is an immersed submanifold of $\Ort(H)$, and furthermore,
$$\Lie_{\phi} \hat H \subseteq T_{\phi} \mathcal{O}_{\phi_0}, \quad \text{ for any } \phi \in \mathcal{O}_{\phi_0}.$$
Furthermore, since $p_* \hat H = H$ and since $H$ is bracket-generating, we have that $p_* \Lie_{\phi} \hat H = T_{p( \phi)} M$. It follows that $p(\mathcal{O}_{\phi_0}) = M$. Hence, for any $y \in M$, there is an isometry $\hat \Phi: \hat{F}_{\phi_0} \to \hat{F}_\phi$ with $\hat \Phi(\phi_0) = \phi$ for some $\phi \in \Ort(H)_y$. As a consequence, there is a local isometry $\Phi$ taking $x$ to $y$.

\item Note that if $R^{g}$ is the curvature of the Levi-Civita connection, then
\begin{align*}
\langle R^{g}(Z_1, Z_2) B_1, B_2 \rangle_{g} & = \langle R^{\nabla}(Z_1, Z_2) B_1, B_2 \rangle_{g} + \frac{1}{2}  (\nabla_{Z_1} \zeta)(Z_2, B_1, B_2)  \\
& \quad - \frac{1}{2} (\nabla_{Z_2} \zeta)(Z_1, B_1, B_2) - \frac{1}{2}\langle  T(Z_1, Z_2), T(B_1, B_2) \rangle_{g} \\
& \quad - \frac{1}{4} \langle T(Z_1, B_2), T(Z_2, B_1) \rangle + \frac{1}{4} \langle T(Z_1, B_1), T(Z_2, B_2) \rangle
\end{align*}
for $Z_j$, $B_j \in \Gamma(TM)$. Since all the leaves of the foliation are totally geodesic, if~$R^{\calF}$ denotes the curvature tensor along the leaves, then $\langle R^{g}(Z_1, Z_2) B_1, B_2 \rangle = \langle R^{\calF}(Z_1, Z_2) B_1, B_2 \rangle$ whenever all vector fields take values in $V$. We compute using any local orthonormal bases $A_1, \dots, A_n$ and $Z_1, \dots, Z_\nu$ of $H$ and $V$, respectively,
\begin{align*}
\quad \Ric_{g}(v,v) &= \sum_{i=1}^n \langle R^{g}(A_i, v)v , A_i \rangle_{g} + \sum_{s=1}^\nu \langle R^{g}(Z_s, \pr_{H} v) \pr_{H} v , Z_s \rangle_{g} \\
& \quad + 2\sum_{s=1}^\nu \langle R^{g}(Z_s, \pr_{V} v) \pr_{H} v , Z_s \rangle_{g} + \sum_{s=1}^\nu \langle R^{g}(Z_s, \pr_{V} v) \pr_{V} v , Z_s \rangle_{g} \\
&= \Ric(\nabla)(\flat v)(v)  + \frac{1}{2} \sum_{i=1}^n \| \calR(A_i, v) \|_{g}^2  + \Ric_{\calF}(\pr_{V} v, \pr_{V} v).
\end{align*}
The result now follows from (a).\qedhere
\end{enumerate}
\end{proof}

\begin{remark} \label{re:geometric}
\begin{enumerate}[\rm (a)]
\item Let $g$ be any metric taming $g_H$ such that $\II = 0$. Write $V$ for the orthogonal complement of $H$. Then for any $\ve >0$, the scaled Riemannian metric
$$g_\ve(v,w) = g(\pr_{H} v, \pr_{H}) + \frac{1}{\ve} g(\pr_{V} v, \pr_{V} w),$$
also tames $g_H$ and satisfies $\II = 0$. Since $\rnabla_A B$ is independent of $g|{V}$ whenever at least one of the vector fields takes values only in $H$, it behaves better with respect to the scaled metric. Such scalings of the extended metric are important for the sub-Riemannian curvature-dimension inequality, see \cite{BaGa17,BaBo12,BBG14,BKW16,GrTh14a,GrTh14b}.
\item If $\bar{\calR} = 0$ then we have that $ \tr_H ( \nabla_\times \calR)(\times, \blank) =   \tr_H( \rnabla_\times \calR)(\times, \blank)$. If this map vanishes, i.e.~if $\Ric(\nabla)$ is a symmetric operator, then $H$ is said to satisfy \emph{the Yang-Mills condition}. One may consider subbundles $H$ satisfying this condition as locally minimizing the curvature $\calR$. See \cite[Appendix A.4]{GrTh14b} for details. \end{enumerate}
\end{remark}

\subsubsection{Regular foliations}
We will give a short remark on the case in Section~\ref{sec:Foliations} when the foliation is also regular, i.e. when there is a global Riemannian submersion $\pi: (M, g) \to (\check{M}, \check{g})$ with the foliation $\calF = \{ F_y = \pi^{-1}(y) \, : \, y \in \check{M} \}$. We can rewrite \eqref{RicGRic} as
\begin{align*} \Ric_{g}( v,v)  &= \Ric(\rnabla)(\flat v)v- \frac{3}{2} \| \calR(v, \blank) \|^2_{g^{*\otimes 2}} + \Ric_\calF(\pr_{V} v, \pr_{V} v)  .
\end{align*}
Hence, requiring that $\Ric(\rnabla)$ is bounded from below is even weaker than requiring this for $\Ric_{g}$. This weaker condition is a sufficient requirement for infinite lifetime for the case of regular foliations.

To prove this, we need a result in \cite{Her60}. Fix a point $y_0\in \check{M}$ and let $\sigma: [0,1] \to \check{M}$ be a smooth curve with $\sigma(0) = y_0$. Define $F = F_{y_0}$ and write $\sigma^x$ for the $H$-horizontal lift of $\sigma$ starting at $x \in F$. Then the map
$$\Psi_{\sigma(t)}: F \to F_{\sigma(t)}, \quad \Psi_{\sigma(t)}(x) = \sigma^x(t) ,$$
is an isometry, so all leaves of $\calF$ are isometric. Write $G$ for the isometry group of~$F$ and $Q_y$ for the space of isometries $q:F \to F_y$. Then $Q = \coprod_{y \in \check{M}} Q_y$ can be given a structure of a principal bundle, such that
$$p: Q \times F \to M \cong (Q \times F)/G, \quad (q, z)  \mapsto q(z).$$
In the above formula, $\phi \in G$ acts on $F$ on the right by $z \cdot \phi = \phi^{-1}(z)$. Finally, if we define
$$E = \left\{  \frac{d}{dt} \Psi_{\sigma(t)} \circ \phi \, : \, \begin{array}{c} \sigma \in C^\infty([0,1],\check{M}) \\
\sigma(0) = y_0, \, \,  \phi \in G, \, \, t \in [0,1] \end{array}  \right\} \subseteq TQ,$$
then $E$ is a principal connection on $Q$ and $p_* E = H$.

One can verify that if $Y_t(y)$ is the Brownian motion in $\check{M}$ starting at $y \in \check{M}$ with horizontal lift $\tilde Y_t(q)$ to $q \in Q_y$ with respect to $E$, then $X_t(x) =p(\tilde Y_t(q), z)$ is a diffusion in $M$ with infinitesimal generator $\frac{1}{2} \Delta_{H}$ starting at $x = p(q,z)$. Hence, if $Y_t(y)$ has infinite lifetime, so does $X_t(x)$ as a process and its horizontal lifts to principal bundles have the same lifetime by \cite{Shi82}. Since a lower bound of $\Ric(\rnabla)$ is equivalent to a lower bound of the Ricci curvature of $\check{M}$ by \cite[Section~2]{GrTh14a}, this is a sufficient condition for infinite lifetime of $X_t(x)$.

The above argument does not depend on $H$ being bracket-generating. However, in the case of $H$ bracket-generating, $F$ is a homogeneous space by a similar argument to that of the proof of Proposition~\ref{prop:LocSym}.

\subsubsection{A counter-example} We will give an example showing that the assumption $\bar{\calR} =0$ is essential for the conclusion of Proposition~\ref{prop:LocSym}.

\begin{example} Consider $\mathfrak{su}(2)$ as the Lie algebra spanned by elements $A_1$, $A_2$ and $A_3$ with bracket-relations
$$[A_1, A_2] = A_3, \quad [A_3, A_1] = A_2, \quad [A_2, A_3] = A_1.$$
Let $G$ be any connected Lie group with Lie algebra $\mathfrak{su}(2)$. Denote the left invariant vector fields and their corresponding elements in the Lie algebra by the same symbol. Let $\varphi: G \to \tilde G$ be a Lie group isomorphism to another copy $\tilde G$ of $G$. Use this to define vector fields on $G \times \tilde G$ by
$$A_k^s(x, y) =A_k(x) + \varphi_* A_k(\varphi^{-1}(y)), \quad A_k^a(x, y) =A_k(x) - \varphi_* A_k(\varphi^{-1}(y)),$$
for any $(x,y) \in G \times \tilde G$. Observe that $[A_i, A_j] = A_k$ implies that
$$[A_i^s, A_j^s] = A_k^s, \quad [A_i^s, A_j^a] = A_k^a, \quad [A_i^a, A_j^a] = A_k^s.$$

Consider $\mathbb{R}$ with coordinate $c$. Define a manifold $M = G \times \tilde G \times \mathbb{R}$. Let $f$ be an arbitrary smooth function on $M$ that factors through the projection to $\mathbb{R}$, i.e.~$f(x,y,c) = f(c)$ for $(x,y, c) \in G \times \tilde G \times \mathbb{R}$. Write $\partial_c f$ simply as~$f'$. Introduce the vector fields $Z_j$, $j = 1,2,3$ on $M$ such that
$$Z_1 = e^f A_1^s, \quad Z_2 = e^f A_2^s, \quad Z_3 = e^f A_1^a.$$
Define a Riemannian metric $g$ on $M$ such that $Z_1$, $Z_2$, $Z_3$, $A_3^s$, $A_2^a$, $A_3^a$, $\partial_c$ form an orthonormal basis. Define a sub-Riemannian manifold $(M, H, g_H)$ such that $H$ is the span of $Z_1$, $Z_2$, $Z_3$ and $\partial_c$ with $g_H$ as the restriction of $g$ to this bundle. If we define~$\II$ and~$\calC$ as in respectively \eqref{II} and \eqref{calC}, then we have $\II =0$ and $\calC = 0$, even though $\bar{\calR}\neq 0$. If $\nabla$ is as in \eqref{MaxRicCurv}, then $\Ric(\nabla)$ is given by

$$\Ric(\nabla) \colon \left\{\begin{aligned}
\flat Z_1 & \mapsto  \left( f''-  e^{2f} (e^{2f} -1)  - 3( f')^2 \right) \flat Z_1, \\
\flat Z_2 & \mapsto \left( f''-  2e^{2f} (e^{2f} -1)  - 3( f')^2 \right) \flat Z_2, \\
\flat Z_3 & \mapsto  \left( f''-  e^{2f} (e^{2f} -1)  - 3( f')^2 \right) \flat Z_3, \\
\flat \partial_c  &\mapsto 2 \left( f''- ( f')^2 \right) \flat \partial_c. 
\end{aligned}\right.$$
%
However, one can also verify that if $\Ric_{g}$ is the Ricci curvature of $g$, then
$$\Ric_{g}(A_2^a, A_2^a) = 2-e^{-f}.$$
Hence, if $f'$ and $f''$ are bounded and $f$ is bounded from above but not from below, then $\Ric(\nabla)$ has a lower bound, but not $\Ric_{g}$. We may for example take $f(c) = - c \tan^{-1}c$. 

%
\end{example}

\section{Torsion and integration by parts} \label{sec:Torsion}
\subsection{Torsion and integration by parts} \label{sec:TorsionRep}
For a function $f \in C^\infty(M)$ on a sub-Riemannian manifold define the horizontal gradient $\nabla^H f = \shh df$. The fact that the parallel transport $\hptr_t$ in Theorem~\ref{th:main} does not preserve the horizontal bundle, makes it difficult to bound $\nabla^H P_t f$ by terms only involving the horizontal part of the gradient of $f$ and not the full gradient. We therefore give the following alternative stochastic representation of the gradient.

Let $(M, g_H^*)$ be a sub-Riemannian manifold and let $\nabla$ be compatible with $g_H^*$. Let $g$ be a Riemannian metric taming $g_H$ and assume that $\nabla$ is compatible with $g$ as well. Introduce a zero order operator
\begin{align} \label{scrA} \mathscr{A}(\alpha) & := \Ric(\nabla)\alpha - \alpha(  \tr_H (\nabla_\times T^\nabla)(\times, \blank)) -\alpha( \tr_H T^\nabla(\times, T^\nabla(\times, \blank))) \\ \nonumber
& = \Ric(\hat{\nabla}) \alpha -\alpha( \tr_H T^\nabla(\times, T^\nabla(\times, \blank))).
\end{align}
Let $X_t(\blank)$ be the stochastic flow of $\frac{1}{2} L(\nabla)$ with explosion time $\tau(\blank)$. Write $\ptr_t = \ptr_t(x): T_xM \to T_{X_t(x)} M$ for parallel transport with respect to $\nabla$ along $X_t(x)$. Observe that this parallel transport along $\nabla$ preserves $H$ and its orthogonal complement. Let $W_t = W_t(x)$ denote the anti-development of $X_t(x)$ with respect to~$\nabla$ which is a Brownian motion in $(H_x, \langle \blank, \blank \rangle_{g_H(x)})$.

\begin{theorem} \label{th:Rep2}
Assume that $\tau(x) = \infty$ a.s. for any $x\in M$~and that for any $f \in C_b^\infty(M)$ with bounded gradient, we have $\sup_{t \in [0, t_1]} \|dP_tf\|_{L^\infty(g^*)} < \infty$. Furthermore, assume that $| T^\nabla |_{\wedge^2 g^* \otimes g} < \infty$ and that $\mathscr{A}$ is bounded from below. Define stochastic processes $Q_t = Q_t(x)$ and $U_t = U_t(x)$ taking values in $\End T_x^*M$ as follows:
$$\frac{d}{dt} Q_t = - \frac{1}{2} Q_t \mathscr{A}_{\ptr_t} \quad Q_0=\id,$$
resp.
$$U_t \alpha(v) =\int_0^t \alpha T^\nabla_{\ptr_s} ( dW_s,  Q_s^\transpose v), \quad T^\nabla_{\ptr_t}(v,w) = \ptr_t^{-1} T(\ptr_t v, \ptr_t w).$$
Then for any $f \in C_b^\infty(M)$, 
\begin{equation} \label{GeneralGrad} dP_t f (x)  = \mathbb{E}\left[ (Q_t + U_t) \ptr_t^{-1} df(X_t(x)) \right] .\end{equation}
\end{theorem}
For a geometric interpretation of $\mathscr{A}$ for different choices of $\nabla$, see Section~\ref{sec:Geometric}. The equality \eqref{GeneralGrad} allows us to choose the connection $\nabla$ convenient for our purposes and gives us a bound for the horizontal gradient on Carnot groups in Section~\ref{sec:Carnot}.

For the proof of this result, we rely on ideas from \cite{DrTh01}. \emph{A multiplication} $m$ of $T^*M$ is a map $m: T^*M \otimes T^*M \to T^*M$. Corresponding to a multiplication and a connection~$\nabla$, we have a corresponding first order operator
$$D^m\alpha = m(\nabla_{\blankdown} \alpha).$$ 
\begin{lemma} \label{lemma:MetricTorsion}
Let $\nabla$ be a connection compatible with $g_H^*$ and with torsion $T$. Define $L = L(\nabla)$, $\Ric = \Ric(\nabla)$ and $T = T^\nabla$. Then for any $f \in C^\infty(M)$,
$$L df - dLf = - 2 D^m df  + \mathscr{A}(df),$$
where $m(\beta \otimes \alpha) = \alpha(T(\shh \beta, \blank))$ and
\begin{equation} \label{scrA2} \mathscr{A}(\alpha) = \Ric(\alpha) - \alpha(  \tr_H (\nabla_\times T)(\times, \blank)) -\alpha( \tr_H T(\times, T(\times, \blank))). \end{equation}
\end{lemma}

\begin{proof}
Recall that if $\hat{\nabla}$ is the adjoint of $\nabla$ and $\hat{L} = L(\hat{\nabla})$, then
$$(\hat{L} df - dL f) = \Ric df.$$
The result now follows from Lemma~\ref{lemma:Adjoiint} and the fact that for any $A \in \Gamma(H)$,
$$\hat{\nabla}_A = \nabla_A + \kappa(A),$$
where $\kappa(A)\alpha = \alpha(T(A, \blank)) = m(\flat A \otimes \alpha)$.
\end{proof}

\begin{proof}[Proof of Theorem~\ref{th:Rep2}]
Let $x \in M$ be fixed. To simplify notation, we shall write $X_t(x)$ simply as $X_t$. Define $\ptr_t$ as parallel transport with respect to~$\nabla$ along~$X_t$. Define $Q_t$ as in Theorem~\ref{th:Rep2}.
For any $S > 0$, consider the stochastic process on $[0,S]$ with values in $T_x^*M$,
$$N_t =\ptr_t^{-1} dP_{S-t} f(X_t).$$
By Lemma~\ref{lemma:MetricTorsion} and It\^o's formula
$$dN_t = \ptr_t^{-1} \nabla_{\ptr_t dW_t} dP_{S-t} f(X_t) - \ptr_t^{-1}D^m d P_{S-t}f(X_t) dt+ \frac{1}{2} \ptr_t^{-1}\mathscr{A}(dP_{S-t}f(X_t)) dt,$$
and so
$$dQ_t N_t = Q_t\ptr_t^{-1} \nabla_{\ptr_t dW_t} dP_{S-t} f(X_t) - Q_t \ptr_t^{-1}D^m d P_{S-t}(X_t) dt.$$
Since $W_t$ is a Brownian motion in $H_x$ and $\ptr_t$ preserves $H$ and its inner product, the differential of the quadratic covariation equals
$$d[ U_t, N_t ]= Q_t \ptr_t^{-1}D^m d P_{S-t}f(X_t) dt.$$
Hence, $(Q_t + U_t) N_t$  is a local martingale which is a true martingale from our assumptions. The result follows.
\end{proof}

\subsection{Geometric interpretation} \label{sec:Geometric}
We will look at some specific examples to interpret Theorem~\ref{th:Rep2} and the zero order operator $\mathscr{A}$ in \eqref{scrA2}.

\subsubsection{Totally geodesic Riemannian foliation and its generalizations}
Assume that condition \eqref{IIzero} holds, so that we are in the case of Section~\ref{sec:IIzero}. Define $\nabla$ as in \eqref{MaxRicCurv} and let $\rnabla$ be as in \eqref{rnabla}. Recall that its torsion $\mathring{T}$ equals $\mathring{T} = - \calR - \bar{\calR}$ and that $\nabla_Z = \rnabla_Z +\mathring{T}^*_{\blankdown} Z$. Using Lemma~\ref{lemma:AS}, we obtain
\begin{align*}
&\langle \mathscr{A} \alpha, \beta \rangle  = \langle \Ric(\nabla) \pr_{H}^* \alpha, \pr_{V}^* \beta \rangle_{g^*} + \langle \Ric(\nabla) \pr_{H}^* \alpha, \pr_{H}^* \beta \rangle_{g^*} \\
& \qquad \qquad \qquad - \alpha( \tr_H (\nabla_\times T)(\times, \sharp \beta)) - \langle \sharp \alpha, \tr_H T(\times, T(\times, \sharp \beta)) \rangle_{g} \\
& = \langle \Ric(\nabla) \pr_{H}^* \alpha, \pr_{H}^* \beta \rangle_{g^*}  +2 \langle \iota_T\alpha, \iota_T \beta \rangle_{\wedge^2 g_H^* } \\
& = \langle \Ric(\nabla) \pr_{H}^* \alpha, \pr_{H}^* \beta \rangle_{g^*}  +2 \langle \iota_{\mathring{T}}\alpha, \iota_{\mathring{T}} \beta \rangle_{\wedge^2 g_H^* }  = \langle \Ric(\rnabla) \pr_{H}^* \alpha, \pr_{H}^* \beta \rangle_{g^*} .
\end{align*}
In particular, $\mathscr{A}$ does not depend on $g|{V}$.

\subsubsection{Lie groups of polynomial growth} \label{sec:LieGroup} 
Let $G$ be a connected Lie group with unit~$\mathbf{1}$ of polynomal growth. 
Consider a subspace $\mathfrak{h}$ that generates all of $\mathfrak{g}$. Equip $\mathfrak{h}$ with an inner product and define a sub-Riemannian structure $(H, g_H)$ by left translation of $\mathfrak{h}$ and its inner product. Let~$g$ be any left invariant metric taming~$g_H$. Let~$\nabla$ be the connection defined such that any left invariant vector field on $G$ is $\nabla$-parallel. Then $\nabla$ is compatible with $g_H^*$ and $g$. 
Let $X_t(\blank)$ be the stochastic flow of $\frac{1}{2} L(\nabla)$, which has infinite lifetime by \cite{HaLe86}. Furthermore, $\| d P_t f \|_{L^\infty(g^*)} < \infty$ for any bounded $f \in C_b^\infty(M)$ by \cite{tEl03}. Hence we can use Theorem~\ref{th:Rep2}.

Let $l_x:G \to G$ denote left multiplication on $G$ and write $x \cdot v := dl_x v$. Notice that since we have a left invariant system, $X_t(x) = x \cdot X_t(\mathbf{1}) =: x \cdot X_t$. Furthermore, parallel transport with respect to $\nabla$ is simply left translation so
$$\ptr_t(x) v = (x \cdot X_t \cdot x^{-1}) \cdot v. $$
If $W_t(x)$ is the anti-development of $X_t(x)$ with respect to $\nabla$ then 
$$W_t(x) = x \cdot W_t(\mathbf{1}) =: x \cdot W_t.$$
As $\nabla$ is a flat connection and since
$$T^\nabla(A_1, A_2) = - [A_1, A_2],$$
for any pair of left invariant vector fields $A_1$ and $A_2$, we have that $\mathscr{A}$ in \eqref{scrA2} equals
$$\mathscr{A} = - \alpha (\tr_H T(\times, T(\times, \blank))).$$
In other words, if we define a map $\psi: \mathfrak{g} \to \mathfrak{g}$, by
\begin{equation} \label{psi} \psi = \tr_{H_{\mathbf{1}}} \ad(\times) \ad(\times),\end{equation}
then
$$\mathscr{A}\alpha = - l^*_{x^{-1}} \psi^* l_x^*\alpha, \quad \alpha \in T_x^*M.$$
Both $\mathscr{A}$ and $T^\nabla$ are bounded in $g$. Hence, we can conclude that for any $v \in \mathfrak{g}$ and $x \in G$,
$$dP_t f(x \cdot v) = \mathbb{E}\left[df\left((x \cdot X_t) \cdot \left(  Q_t^\transpose v +  \int_0^t \ad(Q_s^\transpose v) dW_s \right)\right) \right]$$
where $$ Q_t = \exp\left(- t\psi^*/2 \right).$$
Note that $Q_t$ is deterministic in this case.

\subsection{Carnot groups and a gradient bound} \label{sec:Carnot}
 Let $G$ be a simply connected nilpotent Lie group with Lie algebra $\mathfrak{g}$ and identity $\mathbf{1}$. Assume that there exists a stratification $\mathfrak{g} = \mathfrak{g}_1 \oplus \cdots \oplus \mathfrak{g}_k$ into subspaces, each of strictly positive dimension, such that $[\mathfrak{g}_1, \mathfrak{g}_j] = \mathfrak{g}_{1+j}$ for any $j \geq 1$ with convention $\mathfrak{g}_{k+1} = 0$. Write $\mathfrak{h} = \mathfrak{g}_1$ and choose an inner product on this vector space. Define the sub-Riemannian structure $(H, g_H)$ on $G$ by left translation of $\mathfrak{h}$ and its inner product. Then $(G,H, g_H)$ is called \emph{a Carnot group of step $k$}. Carnot groups are important as they are the analogue of Euclidean space in Riemannian geometry, in the sense that any sub-Riemannian manifold has a Carnot group as its metric tangent cone at points where the horizontal bundle is equiregular. See \cite{Bel96} for details and the definition of equiregular.

Let $(G, H, g_H)$ be a Carnot group with $n = \mathrm{rank} \, H$.
Let $\Delta_{H}$ be defined with respect to left Haar measure on $G$, which equals the right Haar measure since nilpotent groups are unimodular. Consider the commutator ideal $\mathfrak{k} = [ \mathfrak{g}, \mathfrak{g}]= \mathfrak{g}_2 \oplus \cdots \oplus \mathfrak{g}_k$ with corresponding normal subgroup $K$. Define the corresponding quotient map
$$\pi: G \to G/K \cong \mathfrak{h},$$
and write $|\pi|\colon x \mapsto | \pi(x) |_{g_H(\mathbf{1})}$.

It is known from \cite{DrMe05} and \cite{Mel08} that for each $p \in (1, \infty)$, there exists a constant $C_p$ such that $| \nabla^H P_t f|_{g_H} \leq C_p (P_t |\nabla^H f|_{g_H})^{1/p}$ pointwise for any $f \in C^\infty(M)$. We want to give a more explicit form of constants satisfying this inequality.
\begin{theorem} \label{th:GradientB}
Let $\psi$ be defined as in \eqref{psi} and assume that $\psi|{\mathfrak{h}} = 0$.
Let $p_t(x,y)$ denote the heat kernel of $\Delta_{H}$ and define $\varrho(x) = p_1(\mathbf{1}, x)$. Define a probability measure $\mathbb{P}$ on $M$ by $d\mathbb{P} = \varrho d\mu$.
\begin{enumerate}[\rm (a)]
\item Consider the function $\vartheta(x) = n + | \pi |(x) \cdot |\nabla^H \log \varrho |_{g_H}(x)$ and for any $p \in (1, \infty]$,
\begin{equation} \label{constantC} C_p= \left( \int_G \varrho(y) \cdot  \vartheta^q(y) \, d\mu(y) ,\right)^{1/q}, \quad \frac{1}{p} + \frac{1}{q} = 1.\end{equation}
Then the constants $C_p$ are all finite and for any $x \in G$ and $t \geq 0$, we have
\begin{equation*}
| \nabla^H P_t f |_{g_H}(x) \leq C_p (P_t | \nabla^H f |^p_{g_H}(x))^{1/p}, \quad f\in C^\infty(M).
\end{equation*}
Furthermore, $C_2 < n+ (nQ - 2 \Cov_{\mathbb{P}} [ |\pi|^2, \log \rho])^{1/2}$ where $\Cov_{\mathbb{P}}$ is the covariance is with respect to $\mathbb{P}$.
\item Let $Q$ be \emph{the homogeneous dimension} of $G$,
\begin{equation} \label{HomDim} Q = \sum_{j=1}^k j (\mathrm{rank} \, \mathfrak{g}_j) .\end{equation}
For any $n$ and $q \in [2, \infty)$, define
$$c_{n,q} = \left(\frac{2^{(q+n+1)/2} \pi^{(n-1)/2} }{ \sqrt{n} } \frac{\Gamma( \frac{n+ q}{2})}{\Gamma(\frac{n}{2})} \right)^{1/q}.$$
Then for $p \in (2, \infty)$, we have
$$|\nabla^H P_tf| \leq (n+  c_{n,q} \sqrt{Q}) \left(  P_t |df|^p\right)^{1/p}, \quad \frac{1}{q} + \frac{1}{p} = \frac{1}{2}. $$
\end{enumerate}
\end{theorem}
The condition $\psi|\mathfrak{h} = 0$ is actually equal to the Yang-Mills condition in the case of Carnot groups, see Remark~\ref{re:YMCarnot}.
In the definition of $\varrho$, the choices of $t=1$ and $x = \mathbf{1}$ in the definition are arbitrary. For any fixed $t$ and $x$, if we replace $\varrho$ by $\varrho_{t,x}(y) := p_t(x,y)$ in \eqref{constantC}, we would still obtain the same bounds. Taking into account \cite[Cor~3.17]{Mel08}, we get the following immediate corollary.
\begin{corollary}
For any smooth function $f \in C^\infty(G)$ and $t \geq 0$, we have
$$P_t f^2 - (P_t f)^2 \leq t \, C^2_2 P_t | \nabla^H f |_{g_H}^2$$
with $C_2$ as in~Eq.~\eqref{constantC}.
\end{corollary}

Let $g$ be a left invariant metric on $G$ taming $g_H$. Let $\nabla$ be the connection on $M$ defined such that all left invariant vector fields are parallel. As
$$\beta(v) = \tr T^\nabla(v, \blank) = 0,$$
we have that $L(\nabla)^* = L(\nabla)$ by Lemma~\ref{lemma:DualSL}. Furthermore, if $A_1, \dots, A_n$ is a basis of $\mathfrak{g}$, then $L(\nabla)f = \sum_{i=1}^n A_i^2 f$ by \cite{ABGR09}. Before passing to our desired inequalities, we review some facts about Carnot groups.

Let $X_t := X_t(\mathbf{1})$  be a $\frac{1}{2} \Delta_{H}$-diffusion starting at the identity~$\mathbf{1}$ and let $\ptr_t$ denote the corresponding parallel transport along $X_t$ with respect to $\nabla$. Let $\pi\colon G \to \mathfrak{h}$ denote the quotient map.
\begin{enumerate}[\rm (i)]
\item For any $v, w \in H$ we have $\langle v , w \rangle_{g_H} = \langle \pi_* v, \pi_* w \rangle_{g_H(\mathbf{1})}$. Hence we can consider our sub-Riemannian structure as having been obtained by choosing a principal Ehresmann connection $H$ on $\pi$ and lifting the metric on~$\mathfrak{h}$. It follows by \cite[Section~2]{GrTh14a} that $\Delta_{H}$ is the horizontal lift of the Laplacian of $(\mathfrak{h}, \langle \blank, \blank \rangle_{g_H(\mathbf{1})})$ and so we have that $W_t = \pi(X_t)$ is a Brownian motion in the inner product space~$\mathfrak{h}$. Since
$$\pi_* v = \pr_{\mathfrak{h}} x^{-1} \cdot v, \quad v \in T_x G,$$
we can identify $W_t$ with the anti-development of $X_t$.
\item Since $\Delta_{H}$ is left invariant, $X_t(x) := x \cdot X_t$ is a $\frac{1}{2} \Delta_{H}$-diffusion starting at $x$, and $P_tf(x) = P_t(f \circ l_x)(\mathbf{1})$ where $l_x$ denotes left translation. In particular, if $\varrho_t(x) := p_t(\mathbf{1},x)$ then $$p_t(x,y) = \varrho_t(x^{-1} y) .$$
\item \label{item:Dil} Since the Lie algebra $\mathfrak{g}$ has a stratification, for any $s >0$, the map $(\dilation_s)_*: \mathfrak{g} \mapsto \mathfrak{g}$ is given by
\begin{equation} \label{infHom} (\dilation_s)_*A \in \mathfrak{g}_j \mapsto s^j A.\end{equation}
is a Lie algebra automorphism. It correspond to a Lie group automorphism $\dilation_s$ of $G$ since $G$ is simply connected. These automorphisms are called \emph{dilations}.
It can be verified that if $A \in \mathfrak{g}_j$ and we use the same symbol for the corresponding left invariant vector field then
$$A (f \circ \dilation_s) = s^j (Af) \circ \dilation_s.$$
\item \label{item:Scale} As a consequence of Item~\eqref{item:Dil} we have
$$\Delta_{H} (f \circ \dilation_s) = s^2 (\Delta_{H} f) \circ \dilation_s,$$
and hence
$$P_t (f \circ \dilation_s) = (P_{s^2 t} f) \circ \dilation_s.$$
Also, for any function $f$, we have $| df|_{g_H^*} \circ \dilation_s = s^{-1} | d(f \circ \dilation_s) |_{g_H^*}$.
\item \label{item:heatkernel} 
Let $Q$ be the homogeneous dimension of $G$ as in \eqref{HomDim}. By definition $\dilation^*_s \mu = s^Q \mu$, and considering \eqref{item:Scale} the heat kernel has the behavior
$$\varrho_{s^2 t} (\dilation_s(x)) =s^{-Q} \varrho_t(x).$$
\item \label{item:Rep2} Clearly $R^\nabla =0$ and $\nabla T =0$ since the torsion takes left invariant vector fields to left invariant vector fields. Hence, for any left invariant vector field $A$, we have $\mathscr{A}^\transpose A = \psi A$ with $\psi$ as in \eqref{psi}.
If $\psi|{\mathfrak{h}} = 0$,  we can apply Theorem~\ref{th:Rep2}. We obtain that for any $v \in \mathfrak{h}$, then
$$dP_t f(v) = \mathbb{E}\left[\ptr_t^{-1} df(X_t) \big(v+ \ad(W_t) v\big)\right].$$
\end{enumerate}
Theorem~\ref{th:GradientB} now follows as a result of the next Lemma. Note that for any function $f\in C^\infty(M)$, we have 
$| \nabla^H f|_{g_H} = | df|_{g_H^*}$.
\begin{lemma}
Assume that $\psi |{\mathfrak{h}}  =0$. For every $t > 0$, define
$$\vartheta_t = n + |\pi| | d \log \varrho_t |_{g_H^*}.$$
For any $p \in (1, \infty]$, let $q \in [1, \infty)$ be such that $\frac{1}{p} + \frac{1}{q} = 1$ and define
\begin{equation} \label{constantCtp} C_{t, p} = \mathbb{E}\left[ \vartheta_t(X_t)^q \right]^{1/q}.\end{equation}
Then
\begin{enumerate}[\rm (a)]
\item $C_{t,p} = C_{1,p} = C_p$ for any $t >0$.
\item The constants $C_p$ are finite. Furthermore, we have the inequality
$$C_2 \leq  n+ \left( nQ + 2 \int_G (n-|\pi|^2) \varrho \log \varrho \,  d\mu \right)^{1/2} = n+ (nQ - 2 \Cov_{\mathbb{P}} [ |\pi|^2, \log \rho])^{1/2} .$$
\end{enumerate}
\end{lemma}

\begin{proof}
To simplify notation, we write $\langle \blank, \blank \rangle_{L^2(\wedge^jg^*)}$ simply as $\langle \blank, \blank \rangle$ and $r = |\pi|^2$.
\begin{enumerate}[\rm (a)]
\item
We use dilations to prove the statement. Observe that $r \circ \dilation_s = s^2 r$ and that $| d \log \varrho_t | \circ \dilation_s= s^{-1} | d \log \varrho_{t/s^2} |$, and so $\vartheta_t  \circ \dilation_s = \vartheta_{t/s^2}$. It follows that
\begin{align*}
(C_{t,p})^q & =  \int_G  \varrho_t \vartheta^q_t \,d\mu  \stackrel{\dilation_{\sqrt{t}}^*}{=}  \int_G  ( \varrho_t \circ \dilation_{\sqrt{t}}) \left( \vartheta_t  \circ \dilation_{\sqrt{t}} \right)^q t^{Q/2} \,d\mu \\
& =  \int_G  \varrho_1 \vartheta_1^q  \,d\mu = (C_p)^q.
\end{align*}

\item We only need to show that for any $1 < q < \infty$,
$$\int_G \varrho (r^{1/2} | d\log \rho |_{g_H^*} )^q d\mu = \int_G r^{q/2} \varrho^{1-q} | d\varrho |_{g_H^*}^q d\mu < \infty .$$

Define $\mathsf{d}(x) = \mathsf{d}_{g_H}(\mathsf{1}, x)$. Then $\pi$ is distance decreasing, so $r(x) \leq \mathsf{d}(x)^2$. By \cite[Theorem~1]{Var90}, for any $0 < \ve < \frac{1}{2}$ there is a constant $k_\ve$ such that
$$\frac{1}{\varrho(x)} \leq k_\ve \exp\left( \frac{\mathsf{d}^2(x)}{2- \ve} \right).$$
Furthermore, by \cite[Theorem~IV.4.2]{VSC92}, for every $\ve^\prime > 0$ there are constants $k_{\ve^
\prime}$ such that
$$| d\varrho |_{g_H^*}(x) \leq k_{\ve^\prime} \exp\left(- \frac{\mathsf{d}^2(x)}{2+ \ve^\prime} \right).$$
Since we can always find the appropriate values of $\ve$ and $\ve^\prime$ such that
$$\frac{q-1}{q} \leq \frac{2-\ve}{2 + \ve^\prime},$$
it follows that $\int_G r^{q/2} \varrho^{1-q} | d\varrho |_{g_H^*}^q d\mu < \infty$.


Next, define the vector field $D$ by $D f= \frac{d}{ds} ( gf \circ \dilation_{1+s})|_{s=0}$ for any function~$f$. If~$f$ satisfies $f \circ \dilation_{\ve} = \ve^k f$, then by definition $Dg = k g.$ By Item~\eqref{item:heatkernel}, we have
$$\mathrm{div} \, D = Q, \quad (t\Delta_{H}  + D + Q) p_t = (t\Delta_{H} - D^*) p_t = 0.$$
The observation
$$\Delta_{H} (\varrho_t \log \varrho_t) = (\log \varrho_t + 1) \Delta_{H} \varrho_t + \varrho_t | \log \varrho_t |_{g_H^*}^2$$
allows us to compute
\begin{align*}
 (C_2 - n)^2 &\leq \langle r, \varrho | \log \varrho |_{g_H^*}^2 \rangle =  \langle r, \Delta_{H} (\varrho \log \varrho) - (\log \varrho + 1) \Delta_{H} \varrho \rangle \\
& = \langle \Delta_H r, \varrho \log \varrho \rangle + \langle r, (\log \varrho + 1)  D\varrho \rangle + Q \langle r, (\log \varrho + 1) \varrho \rangle\\
& = 2n \langle  \varrho, \log \varrho \rangle - \langle Dr, \varrho \log \varrho  \rangle + Q n\\
& = 2 \langle (n-r)  ,\varrho \log \varrho \rangle + Q n\\
\end{align*}
which equals the covariance, since $\int_M r \rho d\mu = n$.\qedhere
\end{enumerate}
\end{proof}

\begin{proof}[Proof of Theorem~\ref{th:GradientB}]
Again, we write $\langle \blank, \blank \rangle_{L^2(\wedge^jg^*)}$ simply as $\langle \blank, \blank \rangle$ and $r = |\pi|^2$.
\begin{enumerate}[\rm (a)]
\item By left invariance, it is sufficient to prove the inequality in the point $x = \mathbf{1}$.
Let $v \in H_{\mathbf{1}} = \mathfrak{h}$ be arbitrary. We will use Theorem~\ref{th:Rep2} and Item~\eqref{item:Rep2}. For every $x \in G$ we have $\sharp dr(x) = 2 x \cdot \pi(x).$ Let us consider the form $\alpha^v$ defined by $\alpha^v (x)= \flat( x \cdot v)$. Then
\begin{align*}
 dP_tf (v) &= \mathbb{E}\left[\ptr_t^{-1} df(X_t) \left(v-\calR(W_t , v)\right)\right] \\
& = \mathbb{E}[\ptr_t^{-1} df(X_t) (v)] -  \mathbb{E}\left[ df(X_t) \calR(\ptr_t (\pi(X_t) \wedge v))\right] \\
& = \mathbb{E}[\ptr_t^{-1} df(X_t) (v)] - \frac{1}{2} \mathbb{E}\left[ df \calR(\sharp dr , \sharp \alpha^v )(X_t) \right]. 
\end{align*} 
Define $F(A,B) = \flat A \wedge \nabla_B$ and extend $F$ to general sections of $TM^{\otimes 2}$ by $C^\infty(M)$-linearity. Consider $F_H = F(g_H^*)$ and notice that
$$F_H f = d_H f = \pr_H^* df, \quad F_H^2 f = df\calR(\, \cdot \, , \, \cdot \,).$$
Hence 
\begin{align*}
 \mathbb{E}\big[ &\langle df\calR(\sharp dr, \sharp \alpha^v)(X_t) \big] = \langle F_H^2 f, \varrho_t dr \wedge \alpha^v \rangle\\ 
&   = \langle F_H f, F_H^* (\varrho_t dr \wedge \alpha^v) \rangle \\
&  = - \langle d_H f, \iota_{\shh d\varrho_t} dr \wedge \alpha^v \rangle - \langle d_H f, \varrho_t (\Delta_{g_H^*} r) \alpha^v \rangle + \langle d_H f, \varrho_t \nabla_{\shh \alpha} dr \rangle
\end{align*}
since $\nabla \alpha^v = 0$. We use the identities $\Delta_{H} r = 2n$ and $\nabla_A dr = 2\flat \pr_{H} A $ to obtain
\begin{align*}
& \mathbb{E}\left[ \left\langle F_H^2 f, dr \wedge \alpha^v \right\rangle_{g^*} (X_t) \right] 
= - \langle d_H f, \iota_{\shh d\varrho_t} dr \wedge \alpha^v \rangle - 2 (n-1) \langle d_H f, \varrho_t \alpha^v \rangle \\
&= - \mathbb{E}\left[ \left\langle d_H f, \iota_{\shh d \log \varrho_t} dr \wedge \alpha^v \right\rangle_{g^*} (X_t) \right] - 2 (n-1) \mathbb{E}\left[ \ptr_t^{-1} d_H f(X_t)(v) \right]. \end{align*}
Hence, if we define $\mathscr{N}_t\colon T_{\mathbf{1}}^* G \to T_{\mathbf{1}}^* G$ by
$$\mathscr{N}_t\beta =n \beta  + \frac{1}{2} \ptr_t^{-1} \iota_{\sharp dr(X_t)} (d \log \varrho_t(X_t) \wedge \ptr_t \beta),$$
then $dP_t f(v) = \mathbb{E}[\mathscr{N}_t \ptr_t^{-1} df(v)]$ for any $v \in H$.

Observe that $|\mathscr{N}_t \beta_t|_{g_H^*} \leq \vartheta_t | \beta |_{g_H^*}$. Using H\"older's inequality, 
this leads us to the conclusion
\begin{align*}
 | dP_t f|_{g_H^*}(\mathbf{1}) &= \sup_{v \in \mathfrak{h}, |v|_{g_H} =1} dP_tf(v)\\ 
&=  \sup_{v \in \mathfrak{h}, |v|_{g_H} =1} \mathbb{E}[ \mathscr{N}_t \ptr_t^{-1} d f(X_t) ( v )]  \\
&  \leq \mathbb{E}[\vartheta_t^q \circ X_t]^{1/q} \mathbb{E}[  |df|_{g_H^*}^p \circ X_t]^{1/p}\\  
&\leq  C_{t,p} (P_t |df|_{g_H^*}^p (\mathbf{1}) )^{1/p} .
\end{align*}
\item
Using $dP_t(v) = \mathbb{E}[\mathscr{N}_t \ptr_t^{-1} ]$, for $p \in (2, \infty]$, $q \in [2, \infty)$ satisfying
$$\frac{1}{q} + \frac{1}{p} + \frac{1}{2} = 1,$$
we have
\begin{align*}
& | dP_1 f|_{g_H^*}(1) \leq n \mathbb{E}[ | df|_{g_H} (X_1)] + \mathbb{E}[(  |\pi | | \log \varrho |_{g_H^*} |df|_{g_H^*} )(X_1)]\\
& \leq n P_1 | df|_{g_H}  + \mathbb{E}\left[  |\pi |^q(X_1) \right]^{1/q}  \mathbb{E}\left[ | \log \varrho |_{g_H^*}^2(X_1) \right]^{1/2} \mathbb{E}\left[ |df|_{g_H^*}^p(X_1) \right]^{1/p} .
\end{align*}
As was observed in \cite[page 9]{BaBo15},
\begin{align*}
& \mathbb{E}\left[ | \log \varrho |_{g_H^*}^2(X_1) \right] = \int_M \rho | \log \varrho |_{g_H}^2 d\mu = \int_M \left( \Delta_{H} (\varrho \log \varrho) - (\log \varrho + 1) \Delta_{H} \varrho \right) d\mu \\
& = \int_M  (\log \varrho + 1) (D +Q) \varrho \, d\mu = \int_M  D (\varrho \log \varrho ) d\mu +  Q\int_M  (\log \varrho + 1)  \varrho \, d\mu \\
& = \int_M  (D+Q) (\varrho \log \varrho ) d\mu +  Q\int_M \varrho \, d\mu=Q
\end{align*} 
while
\begin{align*}
&\mathbb{E}[|\pi|^q(X_1)] = \mathbb{E}[|W_1|^q]  = \frac{2^{(q+n+1)/2} \pi^{(n-1)/2} }{ \sqrt{n} } \frac{\Gamma( \frac{n+ q}{2})}{\Gamma(\frac{n}{2})}
\end{align*}
The result follows.\qedhere
\end{enumerate}
\end{proof}

\begin{remark} \label{re:YMCarnot}
Consider a Carnot group $(G, H, g_H)$ and let $V$ be the complement of~$V$ defined by left translation of $\mathfrak{g}_2 \oplus \cdots \oplus\mathfrak{g}_k$. Since this is an ideal, we obtain the same subbundle using right translation. Extend the $g_H$ to a Riemannian metric $g$ by defining a right invariant metric on $V$. Then condition \eqref{IIzero} holds, but if $\nabla$ is defined as in \eqref{MaxRicCurv}, then $\Ric(\nabla)$ does not have a lower bound for $k \geq 3$. However, the Yang-Mills condition $\tr_H (\nabla_\times \calR)(\times, \blank) = 0$ of Remark~\ref{re:geometric} equals exactly the condition $\psi|{\mathfrak{h}} = 0$.
\end{remark}

\appendix
\section{Feynman-Kac formula for perturbations of self-adjoint operators} \label{sec:FK}
\subsection{Essentially self-adjoint operator on forms} \label{sec:HsrL}
Let $M$ be a manifold with a sub-Riemannian structure $(H, g_H)$ with $H$ bracket-generating. Consider the rough sub-Laplacian $L = L(\nabla)$ relative to some affine connection $\nabla$ on $TM$. 
Let $g$ be a complete sub-Riemannian metric taming~$g_H$ such that~$\nabla g = 0$. Assume that
$$L^* = L = - (\nabla_{\pr_{H} })^* (\nabla_{\pr_{H} }) .$$
We can then make the following statement for operators of the type $L - \mathscr{C}$ where $\mathscr{C} \in \Gamma(\End(T^*M))$. To simplify notation, we denote $\langle \blank, \blank \rangle_{L^2(\wedge^j g^*)}$ as simply $\langle \blank, \blank \rangle$ for the rest of this section.

\begin{lemma} \label{lemma:SA}
Assume that $\mathscr{C}^* = \mathscr{C}$. If $\mathcal{A} =L- \mathscr{C}$ is bounded from above on compactly supported forms, i.e.~if
$$\lambda_0 = \lambda_0(\mathcal{A}) = \sup\left\{ \frac{\langle \mathcal{A} \alpha, \alpha \rangle}{\langle \alpha, \alpha \rangle} \, \colon \, \alpha \in \Gamma_c(T^*M) \right\} < \infty,$$
then $\mathcal{A}$ is essentially self-adjoint on compactly supported one-forms.
\end{lemma}
We follow the argument of \cite[Section~2]{Str83}. We begin by introducing the following lemma.

\begin{lemma} \cite[Section~X.1]{ReSi75} 
Let $\mathcal{A}$ be any closed, symmetric, densely defined operator on a Hilbert space with domain $\Dom(\mathcal{A})$. Assume that $\mathcal{A}$ is bounded from above by $\lambda_0(\mathcal{A})$ on its domain. Then $\mathcal{A} = \mathcal{A}^*$ if and only if there are no eigenvectors in the domain of $\mathcal{A}^*$ with eigenvalue $\lambda > \lambda_0(\mathcal{A}).$
\end{lemma}

\begin{proof}[Proof of Lemma~\ref{lemma:SA}]
Let $\pr_{H}$ be the orthogonal projection to $H$. Since $L = - (\nabla_{\pr_{H}})^*(\nabla_{\pr_{H}})$, we have $- \langle \mathscr{C} \alpha, \alpha \rangle  \leq  \lambda_0 \langle \alpha, \alpha \rangle$. Denote the closure of $\mathcal{A}| \Gamma_c(T^*M)$ by $\mathcal{A}$ as well.  Assume that there exist a one-form $\alpha$ in $L^2$ satisfying $\mathcal{A}^* \alpha = \lambda\alpha$ with $\lambda > \lambda_0$. By using a trivialization of the cotangent bundle, we see that $L$ is hypoelliptic, which implies that $\alpha$ is smooth. Let~$f$ be an arbitrary function of compact support and write $d_H f = \pr_{H}^* df$. Then
\begin{align*}  \lambda \langle f^2 \alpha, \alpha \rangle &= \langle f^2 \alpha, \mathcal{A}^* \alpha \rangle = \langle \mathcal{A} (f^2 \alpha), \alpha \rangle \\
& = - \langle f^2 \nabla_{\pr_{H} \blankdown} \alpha, \nabla_{\pr_{H} \blankdown} \alpha \rangle - \langle f^2 \mathscr{C} \alpha,  \alpha \rangle - 2\langle f d_H f \otimes \alpha, \nabla_{\pr_{H} \blankdown} \alpha \rangle \\
& \leq - \| f \nabla_{\pr_{H} \blankdown} \alpha \|_{L^2(g^*)}^2 + \lambda_0 \langle f^2 \alpha,  \alpha \rangle  -2 \langle d_H f \otimes \alpha, f \nabla_{\pr_{H} \blankdown} \alpha \rangle.
\end{align*}
Since $(\lambda - \lambda_0) \langle f^2 \alpha,\alpha \rangle  \geq 0$, we have
$$ \left\| f\nabla_{\pr_{H} \blankdown} \alpha \right\|^2_{L^2(g^*)}  \leq  - 2\langle d_H f \otimes \alpha, f\nabla_{\pr_{H} \blankdown } \alpha \rangle,$$
and hence
\begin{equation} \label{ineqSA} \left\| f \nabla_{\pr_{H} \blankdown} \alpha \right\|_{L^2(g^*)}^2  \leq 2 \| d_H f \|_{L^\infty(g^*)} \| \alpha \|_{L^2(g^*)} \| f\nabla_{\pr_{H} \blankdown} \alpha \|_{L^2(g^*)}.
\end{equation}
Since we assumed that $g$ was complete, there exist a sequence of smooth functions $f_j \uparrow 1$ of compact support satisfying $\| df_j\|_{L^\infty(g^*)} \to 0$. By inserting $f_j$ in \eqref{ineqSA} and taking the limit we obtain $\|  \nabla_{\pr_{H} \blankdown} \alpha \|_{L^2(g^*)}^2 = - \langle L \alpha, \alpha\rangle = 0$. However, this contradicts our initial hypothesis $\mathcal{A}^* \alpha = \lambda \alpha$ for $\lambda > \lambda_0$. Hence, we obtain our result. 
\end{proof}

\begin{remark}
By replacing the sequence $f_j$ in the proof of Lemma~\ref{lemma:SA} with (an appropriately smooth approximation of) the sequence found in \cite[Theorem~7.3]{Str86}, we can deduce essential self-adjointness of $L - \mathscr{C}$ just by assuming completeness of~${\sf d}_{g_H}$.
\end{remark}

\subsection{Stochastic representation of a semigroup}
Let $(M, H, g_H)$ be a sub-Rie\-mann\-ian manifold and let $g$ be a complete Riemannian metric taming $g_H$. Define $L^2(T^*M)$ as the space of all one-forms in $L^2$ relative to $g$. Let $\nabla$ be a connection satisfying $\nabla g =0$ and $L^* = L$. Relative to $L(\nabla)$, consider the stochastic flow $X_t( \blank)$ with explosion time $\tau(\blank)$. Define $\ptr_t(x)$ as parallel transport along~$X_t(x)$ with respect to~$\nabla$.

Let $\mathscr{C}$ be a zero order operator on $M$, with
$$\mathscr{C}^s = \frac{1}{2}(\mathscr{C} + \mathscr{C}^*), \quad \mathscr{C}^a = \frac{1}{2}(\mathscr{C} - \mathscr{C}^*).$$
\begin{lemma} \label{lemma:StochRep}
Assume that $L - \mathscr{C}^s$ is bounded from above and assume that $\mathscr{C}^a$ is bounded.
For each $x$, let $Q_t(x) \in \End T_x^*M$ a continuous process adapted to the filtration of $X_t(x)$ such that for any $\alpha \in \Gamma_c(T_x^*M)$, we have
$$d \Big( Q_t(x) \ptr_t^{-1} \alpha(X_t(x)) \Big) \stackrel{\text{\emph{loc.\,m.}}}{=} Q_t(x) \ptr_t^{-1} (L - \mathscr{C})\alpha(X_t(x)) dt,$$ 
where $\stackrel{\emph{\text{loc.\,m.}}}{=}$ denotes equality modulo differentials of local martingales.

 Then there exists a strongly continuous semigroup $P_t^{(1)}$ on $L^2(T^*M)$ such that for any $\alpha \in L^2(T^*M)$,
$$P_t^{(1)}\alpha(x) = \mathbb{E}\left[ 1_{t < \tau(x)} Q_t(x) \ptr_t^{-1} \alpha(X_t)(x) \right],$$
and such that $\lim_{t \downarrow 0} \frac{d}{dt} P_t^{(1)} \alpha = (L - \mathscr{C}) \alpha$ for any $\alpha \in \Gamma_c(TM)$.
\end{lemma} 

For the proof, we need to consider a special class of Volterra operators. To this end, we follow the arguments of \cite[Section III.1]{EnNa00}. Let $\mathfrak{B}$ be a Banach space and let $\mathscr{L}(\mathfrak{B})$ be the space of all bounded operators on $\mathfrak{B}$ with the strong operator topology. Consider any strongly continuous semigroup $\mathbb{R}_{\geq 0} \to \mathscr{L}(\mathfrak{B})$, $t \mapsto S_t$ and let $\mathscr{A}\colon \mathfrak{B} \to \mathfrak{B}$ be a bounded operator. We define the corresponding Volterra operator $\mathsf{V}(S; \mathscr{A})$ on continuous functions $\mathbb{R}_{\geq 0} \to \mathscr{L}( \mathfrak{B})$, $(t, \alpha) \mapsto F_t \alpha$ by
$$(\mathsf{V}(S;\mathscr{A})F)_t \alpha = \int_0^t S_{t-r} \mathscr{A} F_r\alpha \, dr,$$
and introduce the operator $\mathsf{T}(S; \mathscr{A})$ by
$$\mathsf{T}(S; \mathscr{A})F =  \sum_{n=0}^\infty \mathsf{V}(S; \mathscr{A})^n F.$$
The operator $\mathsf{T}(S; \mathscr{A})$ is well defined, and if $S_t$ has generator $(L, \Dom(L))$ then $\tilde S_t := (\mathsf{T}(S; \mathscr{A}) S)_t$ defines a strongly continuous semigroup with generator $(L + \mathscr{A}, \Dom(L))$.

\begin{proof}
By Lemma~\ref{lemma:SA} the operator $L - \mathscr{C}^s$ is essentially self-adjoint. Let $P^s_t$ be the corresponding semigroup on $L^2(T^*M)$ with domain $\Dom^s = \Dom(L- \mathscr{C}^s)$.

Let $D^n$ be an exhausting sequence of $M$ of relative compact domains, see e.g. \cite[Appendix~B.1]{DrTh01} for construction. Consider the Friedrichs extension $(\Lambda^n, \Dom(\Lambda^n))$ of $L- \mathscr{C}^s$ restricted to compactly supported forms on $D^n$ and let $\tilde P^{n}_t$ be the corresponding semigroup defined by the spectral theorem. Since the operators $\Lambda^n$ are bounded from above by assumption, the semigroups $\tilde P^{n}$ are strongly continuous by \cite[Chapter II.3\,c]{EnNa00}. Define $P_t^s$ similarly with respect to the unique self-adjoint extension of $L-\mathscr{C}^s$ restricted to compactly supported forms. Let $(\Lambda, \Dom(\Lambda))$ denote the generator of $P_t^s$ and note that for any compactly supported forms $\alpha$, we have that $\tilde P_t^n \alpha $ converge to $P_t^s \alpha$ in $L^2(T^*M)$, by e.g.~\cite[Chapter VIII.3.3]{Kat95}. Define $P^{n}_t = (\mathsf{T}(\tilde P^n; \mathscr{A})\tilde P^n)_t$ and finally $P_t^{(1)} = (\mathsf{T}( P^s; \mathscr{C}^a) P^s)_t$. These semigroups are strongly continuous with respective generators $(\Lambda^n + \mathscr{C}^a, \Dom(\Lambda^n ))$ and $(\Lambda + \mathscr{C}^a, \Dom(\Lambda))$. Furthermore, $P_t^n\alpha $ converge to $P_t^{(1)}\alpha$ in $L^2(TM)$ by \cite[Theorem IV.2.23 (c)]{Kat95}.

For $x \in M$, let $\tau_n(x)$ denote the first exist time for $X_t(x)$ of the domain~$D^n$. For any form $\alpha$ with support in $D^k$, we have that for $S > 0$ and $n \geq k$,
$$N^n_t = Q_t(x) \ptr_t^{-1} (P^{n}_{S-t} \alpha) |_{X_t(x)}$$
is a bounded local martingale, giving us
$$P^{n}_t\alpha(x) = \mathbb{E}\left[1_{t < \tau^n(x)} Q_t(x) \ptr_t^{-1} \alpha(X_t(x)) \right].$$
Taking the limit, and using that $P^{n}_t$ converges to $P_t^{(1)}$, we obtain
\begin{equation*}
P_t^{(1)} \alpha(x) = \mathbb{E}\left[1_{t<\tau(x)} Q_t(x) \ptr_t^{-1} \alpha(X_t(x)) \right].\qedhere
\end{equation*}
\end{proof}

\bibliographystyle{habbrv}
\bibliography{Bibliography}

\end{document}